\documentclass[12pt]{amsart}

\usepackage[all]{xy}
\usepackage{amssymb}
\usepackage{amsmath}
\usepackage{stmaryrd}
\usepackage{bbm}
\usepackage{txfonts}
\usepackage{tikz-cd}
\usepackage{url}
\usepackage{etoolbox}
\swapnumbers

\usepackage[
    backend=biber,       
    style=numeric,       
    doi=false,            
    eprint=true,         
    url=true             
]{biblatex}
\makeatletter
\patchcmd{\@addmarginpar}{\ifodd\c@page}{\ifodd\c@page\@tempcnta\m@ne}{}{}
\makeatother
\reversemarginpar

\usepackage[OT2,T1]{fontenc}
\DeclareSymbolFont{cyrletters}{OT2}{wncyr}{m}{n}
\DeclareMathSymbol{\Sha}{\mathalpha}{cyrletters}{"58}

\theoremstyle{definition}
\newtheorem*{Proposition*}{Proposition}

\theoremstyle{definition}
\newtheorem*{Theorem*}{Theorem}

\theoremstyle{definition}
\newtheorem*{Corollary*}{Corollary}

\theoremstyle{definition}
\newtheorem*{Conjecture*}{Conjecture}

\theoremstyle{definition} 
\newtheorem{ssProposition}[subsubsection]{Proposition}

\theoremstyle{definition}

\theoremstyle{definition} 

\newtheorem{ssDefinition}[subsubsection]{Definition}

\theoremstyle{definition}

\theoremstyle{definition}

\theoremstyle{definition}

\theoremstyle{definition}

\theoremstyle{definition}

\theoremstyle{definition} 
\newtheorem{ssTheorem}[subsubsection]{Theorem}

\theoremstyle{definition} 
\newtheorem{ssLemma}[subsubsection]{Lemma}

\theoremstyle{definition}

\theoremstyle{definition}

\theoremstyle{definition} 
\newtheorem{ssCorollary}[subsubsection]{Corollary}

\theoremstyle{definition}

\theoremstyle{definition} 
\newtheorem{ssRemark}[subsubsection]{Remark}

\theoremstyle{definition} 
\newtheorem{ssExample}[subsubsection]{Example}

\theoremstyle{definition} 
\newtheorem{ssExamples}[subsubsection]{Examples}

\theoremstyle{definition}

\theoremstyle{definition}

\theoremstyle{definition}

\theoremstyle{definition}

\let\oldmarginpar\marginpar
\renewcommand\marginpar[1]{\-\oldmarginpar[\raggedleft\footnotesize #1]%
{\raggedright\footnotesize #1}}

\newcommand{\set}[2]{\left\{ #1 \; \middle| \; #2 \right\} }

\newcommand{\Rep}{\operatorname{\bf{Rep}}}
\newcommand{\Vect}{\operatorname{\bf{Vect}}}

\newcommand{\from}{\leftarrow}
\newcommand{\xto}{\xrightarrow}
\newcommand{\xfrom}{\xleftarrow}
\newcommand{\surj}{\twoheadrightarrow}

\newcommand{\cok}{\operatorname{cok}}
\newcommand{\gr}{\operatorname{gr}}

\renewcommand{\Im}{\operatorname{Im}}

\newcommand{\Spec}{\operatorname{Spec}}

\newcommand{\Hom}{\operatorname{Hom}}
\newcommand{\Ext}{\operatorname{Ext}}

\newcommand{\Lie}{\operatorname{Lie}}

\newcommand{\Mod}{\operatorname{Mod}}

\newcommand{\m}[1]{\mathrm{#1}}
\newcommand{\fk}[1]{\mathfrak{#1}}

\newcommand{\bb}[1]{\mathbb{#1}}
\newcommand{\cl}[1]{\mathcal{#1}}

\newcommand{\la}{\lambda}
\newcommand{\ka}{\kappa}
\newcommand{\si}{\sigma}

\newcommand{\ga}{\gamma}
\newcommand{\al}{\alpha}
\newcommand{\be}{\beta}
\newcommand{\om}{\omega}

\newcommand{\ep}{\epsilon}
\newcommand{\de}{\delta}

\newcommand{\GL}{\operatorname{GL}}

\newcommand{\Gm}{{\mathbb{G}_m}}
\newcommand{\Qp}{{\QQ_p}}

\newcommand{\Cc}{\mathcal{C}}

\newcommand{\ZZ}{\bb{Z}}
\newcommand{\CC}{\bb{C}}
\newcommand{\GG}{\mathbb{G}}

\newcommand{\NN}{\bb{N}}
\newcommand{\Ss}{\cl{S}}
\newcommand{\QQ}{\bb{Q}}
\newcommand{\RR}{\bb{R}}
\newcommand{\PP}{\bb{P}}
\newcommand{\pP}{\fk{p}}

\newcommand{\Hh}{\mathcal{H}}
\renewcommand{\AA}{\bb{A}}

\newcommand{\Vv}{\mathcal{V}}

\newcommand{\Oo}{\mathcal{O}}
\newcommand{\Kk}{\mathcal{K}}
\newcommand{\Tt}{\mathcal{T}}
\newcommand{\Aa}{\mathcal{A}}
\newcommand{\Ee}{\mathcal{E}}

\newcommand{\Dd}{\mathcal{D}}

\newcommand{\Xx}{\mathcal{X}}

\newcommand{\fkS}{\fk{S}}
\newcommand{\fkM}{\fk{M}}

\newcommand{\Cpx}{\operatorname{Cpx}}

\newcommand{\inv}{^{-1}}

\newcommand{\areq}{\ar@{=}}
\newcommand{\suphook}{\ar@{^(->}}
\newcommand{\subhook}{\ar@{_(->}}

\newcommand{\smses}[6]
{
\[
\xymatrix{
1 \ar[r] &
#1 \ar[r]_-{#2} &
#3 \ar[r]_-{#4} &
#5 \ar[r] \ar@/_1.5pc/[l]_-{#6} &
1
}
\]
}

\newcommand{\thrpl}{\PP^1 \setminus \{0,1,\infty\}}

\newcommand{\Fphi}{{\m{F}\phi} }

\newcommand{\dR}{{\rm {dR}}}

\newcommand{\et}{{\textrm {\'et}}}

\newcommand{\PSh}{\operatorname{PSh}}

\newcommand{\DM}{\operatorname{DM}}

\newcommand{\Sm}{\operatorname{Sm}}
\newcommand{\op}{^\m{op}}

\newcommand{\one}{\mathbbm{1}}

\newcommand{\cdga}{\operatorname{cdga}}

\newcommand{\Ind}{\operatorname{Ind}}

\newcommand{\colim}{\operatorname{colim}}

\newcommand{\Set}{\operatorname{Set}}

\newcommand{\Aug}{\operatorname{Aug}}

\newcommand{\opnm}{\operatorname}

\newcommand{\CAlg}{\opnm{CAlg}}

\newcommand{\Map}{\opnm{Hom}}

\newcommand{\Sym}{\opnm{Sym}}

\newcommand{\Fib}{\opnm{Fib}}

\newcommand{\new}{\newcommand}

\new{\spar}
[1]{\medskip\refstepcounter{subsection}\noindent\arabic{section}.\arabic{subsection}.\label{#1}}

\new{\sspar}{\subsubsection}

\new{\oo}{$\infty$}

\new{\Aaug}{\opnm{\Aa\textit{ug}}}
\new{\Ddev}{\opnm{\Dd\textit{ev}}}

\new{\SH}{\cl{SH}}

\new{\Aff}{\opnm{Aff}}

\new{\vectorspace}{{\textit{vector space}}}

\new{\MFphi}{{\m{MF}^\phi}}

\new{\HoT}{\mathbf{Ho\hspace{-.1em}T}}
\new{\SHoT}{\mathbf{SHo\hspace{-.1em}T}}

\new{\PrLst}{\m{Pr}^L_\m{st}}

\new{\fkI}{\mathfrak{I}}

\new{\Ff}{\mathcal{F}}

\new{\qandq}{\quad \text{and} \quad}

\new{\Sch}{\m{Sch}}
\new{\Schft}{\m{Sch}^\m{ft}}
\new{\SchftZ}{\m{Sch}^\m{ft}_Z}
\new{\SchftZop}{(\m{Sch}^\m{ft}_Z)^\m{op}}

\new{\Hhom}{\Hh\hspace{-.3ex} om}

\new{\VV}{\mathbb{V}}
\new{\EE}{\mathbb{E}}

\new{\vectorspaces}{\textit{vector spaces}}

\new{\La}{\Lambda}

\new{\coeq}{\opnm{coeq}}

\new{\Alg}{\opnm{Alg}}
\new{\LMod}{\opnm{LMod}}
\new{\coAlg}{\opnm{coAlg}}

\new{\Tw}{\opnm{Tw}}
\new{\dDD}{\mathfrak{D}}
\new{\DKsz}{\dDD_\m{Ksz}}

\new{\Nn}{\mathcal{N}}

\new{\SHom}{\opnm{SHom}}

\title{Coarse moduli of motivic augmentations}

\author{Ishai Dan-Cohen}
\address{Ben Gurion University of the Negev}
\email{ishaida@bgu.ac.il}

\thanks{\textbf{Grant acknowledgements:} This work was supported by ISF grants 726/17 and 621/21, and by zukunft.niedersachsen, the joint science funding program of the
Lower Saxony Ministry of Science and Culture and the Volkswagen
Foundation. \textbf{Political disclaimer:} I ask not to be considered responsible for the actions of any government that does not fully embrace the principles of democracy. }

\date{\today}

\begin{document}

\begin{abstract}
A variety $X$ over a suitable base $Z$ gives rise to a highly structured algebra $C^*(X)$ in motives over $Z$. In turn, a $Z$-point gives rise to an augmentation $C^*(X) \to \one$. This assignment $X(Z) \to \Aug \big( C^*(X) \big)$ factors the so-called ``unipotent Kummer map'' to torsors under the unipotent fundamental group in realizations. In one direction, this suggests the possibility of \textit{performing} Chabauty-Kim theory motivically without waiting for a motivic t-structure. In a different (largely independent) direction, we may hope to extract arithmetic information for use in bounding sets of integral points from the full rational homotopy type going beyond $\pi_1$. In both directions, the coarse space for motivic augmentations $\Aug \big( C^*(X) \big)$ would benefit from a structure of finite type $\QQ$-variety, and similarly, the coarse space $\Aug \big( C^*_{F\phi}(X) \big)$ of augmentations in filtered $\phi$ modules would benefit from a structure of finite type $\Qp$-variety. We establish two criteria for representability, compute these spaces in several examples, and construct a comparison with Selmer varieties. Finally, we demonstrate how these constructions lead to K-theoretic finiteness criteria in an example. 
\end{abstract}

\maketitle

\setcounter{tocdepth}{2}
\tableofcontents

\raggedbottom
\SelectTips{cm}{11}

\section{Introduction}
\label{Sec:Int}

A variety $X$ over a base scheme $Z$ gives rise to a highly structured algebra $C^*(X)$ in the stable \oo-category $\Dd_M = \DM(Z, \QQ)$ of mixed motivic complexes \cite{Iwanari, rmps, mrmht, DCBetts}\footnote{Unfortunately, the published version of \cite{rmps} contains many errors of grammar and idiom. We recommend the reader consult the arXiv preprint \cite{rmps_arXiv}}. When $X$ admits a $\CC$-point, the associated real Betti realization $C^*_\m{Bett}(X)_\RR$ may be identified with the cdga of smooth forms on $X(\CC)$. Better: the \textit{rational} Betti realization $C^*_\m{Betti}(X)_\QQ$ may be identified with the polynomial de Rham complex of $X(\CC)$ of classical rational homotopy theory \cite{Sullivan}. Thus, $C^*(X)$ may be viewed as the rational homotopy type of $X(\CC)$ equipped with extra motivic structures (mixed Hodge structure, Galois action, crystalline Frobenius, ...). It has long been known that (after fixing a suitable base point) one may retrieve the rational prounipotent completion of the fundamental group of $X(\CC)$ from the cdga of polynomial differential forms, so $C^*(X)$ also remembers the former, along with extra motivic structures. 

These highly structured algebras are somewhat magical. For instance, when $Z = \Spec \ZZ$ and $X = \thrpl$, the underlying object 
\[
C^*(X) \simeq \QQ(0) \oplus \QQ(-1)[-1]^{\oplus 2}
\]
is merely a sum of shifts of Tate motives. Moreover, the multiplication map
\[
\big( \QQ(0) \oplus \QQ(-1)[-1]^{\oplus 2} \big)
\otimes
\big( \QQ(0) \oplus \QQ(-1)[-1]^{\oplus 2} \big)
\to
\QQ(0) \oplus \QQ(-1)[-1]^{\oplus 2}
\]
is (in a sense which can be made precise) trivial. Yet the higher homotopies hidden in its algebra structure essentially remember the periods of the unipotent fundamental group, i.e. multiple zeta values. 

Since $C^*(X)$ is contravariant in $X$, and since $C^*(Z) = \one$ is the trivial algebra, a section of the structure morphism $X \to Z$ gives rise to an augmentation $C^*(X) \to \one$. When $Z$ is the spectrum of a number field $K$ or an open subscheme of $\Spec \Oo_K$ (an ``open integer scheme''), the resulting map
\[
X(Z) \to \Aug\big(C^*(X)\big)
\]
factors the ``unipotent Kummer map'' to the set of torsors under the prounipotent completion of $\pi_1$ used in Chabauty-Kim theory to bound the set $X(Z)$ of $Z$-valued (``$Z$-integral'', ``$Z$-rational'') points. 

There are several reasons to try to develop a version of Chabauty-Kim theory based on augmentations in place of $\pi_1$-torsors. According to a conjecture of Deligne from the 80's \cite{Deligne89}, when the base $Z$ admits a $\CC$-valued point, and $X$ is equipped with a $Z$-integral base-point, the prounipotent completion of the Betti fundamental group of $X(\CC)$ \textit{is motivic}, i.e. comes from a prounipotent group object (called \emph{the unipotent fundamental group}) in the Tannakian category of mixed motives, by applying Betti realization. Indeed, our ability to retrieve the prounipotent completion of the topological fundamental group of $X(\CC)$ along with extra motivic structures functions as a conditional proof of this conjecture. 

Unfortunately, we don't have a fully functional Tannakian (or abelian) category of mixed motives, and instead of \emph{the} unipotent fundamental group itself, practitioners of Chabauty-Kim theory must make do with its \emph{realizations} in the various Weil cohomology theories ($\QQ$-Hodge, $\Qp$-\'etale, filtered-$\phi$,...). By contrast, the space of augmentations is available now, with no need to wait for the elusive motivic t-structure. 

Thus, in one direction, our purpose here is to tighten the connection between motives and arithmetic by short-circuiting the space of Galois equivariant $\pi_1$-torsors of Chabauty-Kim theory. As a sample application in this direction, we're able to obtain an (in principle) effective finiteness result: let $Z$ be the complement of $s$ primes in $\Spec \ZZ$, let $X \to Z$ be the complement of the zero-section in an elliptic scheme of Mordell-Weil rank $r$, and let
\[
k^{(2)}_2(X) = \dim_\QQ K^{(2)}_2(X),
\]
the dimension of an Adams eigenspace in a rationalized Quillen $K$-group. If 
\[
r+s+k^{(2)}_2(X) < 4
\]
then $X(Z)$ is finite.\footnote{If we don't ask for effectivity, then this is just a special case of Siegel's theorem. And even if we do ask for effectivity, then we can turn to ``classical'' Chabauty-Kim theory, which would produce a similar result, with the K-group replaced by an analogous Galois cohomology group. The novelty here is the direct route from K-theory to arithmetic.}

In a different direction, working with augmentations means retaining more information than the associated $\pi_1$-torsors. This information may be particularly valuable for higher dimensional varieties. For the most part, curves over $Z$ whose sets of $Z^o$-integral points for open $Z^o \subset Z$ are \emph{degenerate} (that is, contained in a proper subvariety) have nonabelian $\pi_1$.\footnote{Technically speaking, there is one counterexample.} But this is far from the case for higher dimensional varieties. Thus, to have any hope of applying the methods of Chabauty-Kim theory to higher dimensional varieties, we must wean ourselves from our dependence on $\pi_1$. 

While these directions may be largely independent of one another, they do have something in common: both would benefit significantly from the representability of augmentation spaces. To be precise, let $\Cc_M = \CAlg \Dd_M$ be the \oo-category of commutative algebras in $\Dd_M$. We can associate to $C^*(X)$ a presheaf of homotopy types on the category of affine $\QQ$-schemes
\[
\Aaug \big( X \big):
\Aff_\QQ\op
\to
\HoT
\]
such that
\[
\Aaug \big( X \big)(R) = 
\Hom_{\Cc_M} \big( C^*(X), R \otimes \one \big)
\]
for any $\QQ$-algebra $R$. Then our provisional goal is to understand when the presheaf of connected components
\[
\Aug(X) = \pi_0 \Aaug(X)
\]
is representable by a finite-type $\QQ$-scheme. Results of Ancona et al. \cite{AnconaHuberII} and Iwanari \cite{Iwanari} combined with joint work with L. Alexander Betts \cite[Theorem 7.2]{DCBetts} may be used to show that this is sometimes literally the case:

\begin{ssTheorem}
\label{IN11}
Let $Z$ be an open subscheme of $\Spec \Oo_K$ for $K$ a number field, let $A \to Z$ be an abelian scheme, and let $L^\star \to A$ be a $\Gm$-bundle. Then the augmentation space is a trivial torsor
\[
\xymatrix{
\VV^\lor \Gm(Z)_\QQ
\ar@{}[r]|-{\rotatebox{-90}{\scalebox{2}{$\circlearrowright$}}}
&
\Aug(L^*) \ar[d]
\\
& \VV^\lor A(Z)_\QQ
}
\]
under a vector-group over a vector-group.
\end{ssTheorem}

\noindent
Theorem \ref{IN11} is based on a determination of the structure of $C^*(L^\star)$. The latter instantiates the expectation that highly structured algebras in motives with rational coefficients should form a reasonably computable flavor of unstable motivic homotopy, and this may be viewed as a third (if somewhat less precise) purpose for the present work, namely, to promote motivic rational homotopy theory. Despite its comparative simplicity, motivic rational homotopy theory has particularly direct (and possibly far reaching) implications for arithmetic. 

See Section \ref{sSec:GmBun} for the full statement and proof of the theorem.

But just as one has to cut down the unipotent fundamental group in Chabauty-Kim theory (e.g. by passing to quotients along the descending central series), here too, we usually expect to have to filter $C^*(X)$. I learned about one way to do so from the work of I. Iwanari \cite{Iwanari}, who lifts the Sullivan filtration of classical rational homotopy theory to the motivic setting. I find his explicit computations of the first few steps $\fkS_iC^*(X)$ in the filtration in a range of examples tremendously charming, and yet another purpose of the present work is simply to celebrate Iwanari's computations by exploring their direct implications for augmentation spaces.

Let $\Aaug^i(X) = \Aaug \big(\fkS_i C^*(X) \big)$. We obtain two criteria for the representability of the truncated augmentation spaces
\[
\Aug^i(X) := \pi_0 \Aaug^i(X), 
\]
a \textit{coarse} criterion, and a \textit{refined} criterion. By construction, for each $i \ge 0$, there's a fiber sequence
\[
\Aaug^{i+1}(X) \to \Aaug^i(X) \to \Vv^i
\]
where $\Vv^i$ is the augmentation space 
\[
\Vv^i = \Aaug(\Sym V_i)
\]
of the free commutative algebra on an object $V_i \in \Dd_M$. Our coarse criterion has the advantage of being formulated entirely in terms of the ``linear'' spaces $\Vv^i$. 

\begin{ssProposition}
\label{IN12}
In addition to certain finiteness conditions on the $V_i$, assume
\[
\pi_2 \Vv^i(\QQ) = 0
\quad  \text{for } 0 \le i \le n-2.
\] 
Then $\Aug^n(X)$ is representable by a finite type affine $\QQ$-scheme with dimension bound given by a sum of dimensions of vector spaces coming from K-theory:
\[
\dim \Aug^n(X) 
\le 
\sum_{i=0}^{n-1} 
\dim \pi_0\Hom_{\Dd_M}\big( V_i[1], \QQ(0) \big).
\]
\end{ssProposition}

\noindent
See Proposition \ref{BB37} for the precise statement. 

Unfortunately, while the coarse criterion is enough to work out a few first examples (including e.g. $\Aug^3$ of an elliptic scheme as above), the vanishing of $\pi_2$ fails in other simple examples. This is the case, for instance, for $\Aug^2$ of $\PP^2$ minus four lines in general position (nickname ``four line complement''). We're nevertheless able to establish the representability in this example over an open integer scheme $Z \subset \Spec \Oo_K$ (Section \ref{sSec:4L}). If the fiber $L^i$ of 
\[
\pi_0\Aaug^i(X) \to \pi_0 \Vv^i
\]
is representable (which happens e.g. if the source and target are representable), then $\pi_0 \Aaug^{i+1}(X) \to L^i$ is a trivial torsor under a presheaf of abelian groups $Q^i \to L^i$ over $L^i$. Roughly speaking, $Q^i$ is the presheaf cokernel of the map of presheaves of fundamental groups 
\[
\text{`` }\pi_1\Aaug^i(X) \to \pi_1 \Vv^i\text{ ''}
\]
at a family of base points lifting the identity map of $L^i$; see Proposition \ref{BB44} for the precise construction. Thus, our refined criterion states that under certain finiteness conditions on the $V_i$, if the presheaves of groups $Q^0, Q^1, ..., Q^i$ are representable, then so is $\pi_0\Aaug^{i+1}(X)$; see Corollary \ref{BB45}. 

The homotopy groups $\pi_j \Aaug^i(X)$ have a tendency to be unipotent; see e.g. Proposition \ref{KK61}. Hence, so do the groups $Q^i$. This provides some justification for our choice to work with presheaves in place of sheaves (for a suitable topology): our view is that if the presheaves $\pi_0 \Aaug^i(X)$ become representable after sheafification, then they would be representable (hence in particular sheaves) in the first place. Proposition \ref{KK61} and Remark \ref{VF44} also provide an inkling of how one might go about verifying our refined criterion in a concrete example. The homotopy groups $\pi_j\Aaug^i(X)$ $, \pi_j \Vv^i$ (at a suitable choice of compatible base-points) form an exact couple, equivalently, fall into a grid of knight's moves (see paragraph \ref{BB35}). One may work one's way up through the rows of homotopy groups, and across each row, with each successive representability result providing a useful step towards the representability of the next. The four line complement completes this process in one step: we find that $Q^1 = 0$, hence that
\[
\pi_0\Aaug^2(X) \to \pi_0\Aaug^1(X)
\]
is an isomorphism. Since the target is representable by a vector group, so is the source. 

Suppose $Z \subset \Spec \Oo_K$ is an open integer scheme and $X \to Z$ is the complement of an \'etale divisor in a smooth proper $Z$-scheme equipped with a $Z$-integral base point $x$. Fix a closed point $\pP \in Z$ lying over the prime $p \in \ZZ$ and let $Z^o \subset Z$ be the complement of the set of all primes over $p$. The $\Qp$-\'etale realization $\pi_1^{[n]}(X,x)_\Qp^\et$ of the quotient of the unipotent fundamental group by the $n$th step in the descending central series comes equipped with an action of the Galois group $\pi^\et_1(Z^o)$. The pointed set
\[
H^1_f \big( 
\pi_1^\et(Z^o), 
\pi_1^{[n]}(X,x)_\Qp^\et
\big)
\]
of nonabelian cohomology classes which are crystalline at primes above $p$ may be upgraded in a natural way to a presheaf of pointed sets, known to be representable by a finite type affine $\Qp$-variety under mild assumptions \cite{kimii} (we replace the ordinary `$H$' by a fold-face `$\mathbf{H}$'). 

Let $K_\pP$ be the local field at $\pP$, let $\Oo_\pP$ be its ring of integers, and let $Z_\pP = \Spec \Oo_\pP$ be the associated trait. The $p$-adic de Rham realization $\pi_1^{[n]}(X,x)^\dR_{K_\pP}$ has the structure of a prounipotent group object in the $\Qp$-Tannakian category $\MFphi = \MFphi(K_\pP)$ of admissible filtered $\phi$ modules over $K_\pP$. The pointed set
\[
H^1 \big( \MFphi, \pi_1^{[n]}(X,x)^\dR_{K_\pP} \big)
\]
of $\pi_1^{[n]}(X,x)^\dR_{K_\pP}$-torsors in $\MFphi$ may be upgraded in a natural way to a presheaf of pointed sets (see Definition \ref{DA40}), and under mild assumptions, it too is representable by a finite type affine $\Qp$-variety \cite{kimii}.

Chabauty-Kim theory \cite{kimi, kimii, nats} revolves around a commuting diagram known as ``Kim's cutter''
\[
\xymatrix{
X(Z) \ar[r] \ar[d] & X(Z_\pP) \ar[d] 
\\
\mathbf{H}^1_f \big( 
\pi_1^\et(Z^o), 
\pi_1^{[n]}(X,x)_\Qp^\et
\big) \ar[r]
&
\mathbf{H}^1 \big( \MFphi, \pi_1^{[n]}(X,x)^\dR_{K_\pP} \big)
}
\]
in which the bottom map is a map of presheaves. We show that our construction is compatible with Chabauty-Kim theory in the following sense. There's an evident version of our augmentation presheaves in any realization. In particular, there's a filtered $\phi$ version $\Aaug_{\Fphi}^i(X)$.\footnote{In order to have a realization functor into filtered $\phi$ modules we restrict to lisse motives; see paragraph \ref{EA23}.} 

\begin{ssTheorem}
\label{IN13}
There's a commuting diagram
\[
\xymatrix{
X(Z) \ar[r]^-{\la^g} \ar[d] & X(Z_\pP) \ar[d]^{\de_\Fphi}
\\
\pi_0\Aaug^n(X)_\Qp \ar[d] \ar[r]^-{\la^a}
&
\pi_0\Aaug^n_{\Fphi}(X) \ar[d]^-{\ka_\Fphi}
\\
\mathbf{H}^1_f \big(\pi_1^\et(Z^o), \pi_1^{[n]}(X,x)^\et_\Qp \big) \ar[r]^-{\la^t}
&
\mathbf{H}^1 \big( \MFphi, \pi_1^{[n]}(X,x)^\dR_{K_\pP} \big)
}
\]
factoring Kim's cutter. The lower square consists of maps of presheaves. 
\end{ssTheorem}

\begin{proof}
Combine Proposition \ref{DX66} with Variant \ref{DA39}, plus Tur Dorvault \cite{tur2026motivic} for the case of tangential base points. For the comparison between the $\Qp$-\'etale unipotent fundamental group as defined here and the version used by Kim, see \cite{rmps, dan2017rationalCorr}; combining the latter with \cite{DegliseNiziol} produces a similar comparison of $p$-adic de Rham unipotent fundamental groups.
\end{proof}

Under mild assumptions, the composition $\ka_\Fphi \circ \de_\Fphi$ is known to be locally analytic, given in coordinates by $p$-adic iterated integrals. The method of Chabauty-Kim proceeds as follows. If one can establish the existence of a function
\[
f: \mathbf{H}^1 \big( \MFphi, \pi_1^{[n]}(X,x)^\dR_{K_\pP} \big)
\to 
\AA^1_\Qp
\]
such that
\[
f \circ \la^t = 0,
\]
then we obtain a locally analytic function $f\circ \ka_\Fphi \circ \de_\Fphi$ on $X(Z_\pP)$ that vanishes on $X(Z)$. The latter may thus be used to bound $X(Z)$ inside $X(Z_\pP)$. An optimistic outlook for this project would be to use the upper square in Theorem \ref{IN13} to produce new locally analytic functions vanishing on $X(Z)$. In the case of the four line complement, this could lead to a ``Kim-theoretic'' proof of the degeneracy of the integral points, currently known to follow from the subspace theorem \cite[Theorem 1.2.4]{corvaja2016integral}. Just beyond this example lies the ``$\ominus$-complement'' $\Spec \ZZ[x,y, y\inv, (x^2+y^2-1)\inv]$, whose sets of integral points are conjectured \cite[\S1.2, \textit{Vojta's Conjecture}]{corvaja2016integral} but apparently not known to be degenerate. The theory of $p$-adic iterated integrals of Coleman and Besser \cite{Besser, besser2011heidelberg} may be viewed as a $p$-adic avatar of integration on $n$-simplices on the $n$-fold product $X(\CC)^n$. Here one might expect a notion of integration on higher simplices. 

Apart from the work by Iwanari \cite{Iwanari}, this work is also closely related to work of Pridham \cite{pridham2020nonabelian}. Pridham works primarily in the $p$-adic \'etale setting. He too studies filtrations on homotopy types with associated exact couples of homotopy groups. His main focus is on obstruction spaces analogous to our $\Vv^i$, and on a comparison of \textit{local} and \textit{global} which leads in the direction of Kim's nonabelian reciprocity laws \cite{kim2015diophantine}. 

\bigskip

\noindent
\textbf{Acknowledgements.}
I wish to thank Pietro Corvaja, Isamu Iwanari, Shane Kelly, Minhyong Kim, Jon Pridham, and Amos Turchet for helpful conversations and email exchanges. I wish to thank Sa'ar Zehavi for his energizing enthusiasm concerning the topic of this article. I would particularly like to thank Oliver R\"ondigs for his interest and encouragement, and for his great hospitality during a brief visit to Osnabr\"uck.

\medskip
\subsection{Notation and Terminology}
\label{notation}
Our main references for \oo-categories and higher algebra are Lurie's works \cite{LurieTopos, LurieHA}, whose notation and terminology we follow, with the usual exception that we allow ourselves to identify an ordinary category with the associated \oo-category, and two less common exceptions: we use $\HoT$ for the \oo-category denoted $\Ss$ in loc. cit. and we refer to its objects as \textit{homotopy types}, and we use $\SHoT$ for the stable \oo-category denoted $\Ss p$ in loc. cit. whose objects we refer to variously as \textit{spectra} or \textit{stable homotopy types}.

Our use of (\oo)-category theory is quite tame by modern standards. Thus, we elect to avoid fixing a universe and distinguishing between \textit{large} and \textit{small}. Our view is that correcting the resulting errors caused by this oversight would amount to a routine exercise in adding set theoretic decorations and caveats, while carrying out this exercise within the body of our text would somewhat cloud the exposition. 

Given objects $X$, $Y$ in an \oo-category $\Cc$, we denote the homotopy type of morphisms $X \to Y$ by $\Hom_\Cc(X,Y) \in \HoT$. An object $X$ of $\HoT$ is a simplicial set, and by applying geometric realization, we may think of $X$ as a topological space. When doing so, it's important to bear in mind that many \oo-categorical constructions produce objects that are only well defined up to essentially unique equivalence. 

By an \emph{open integer scheme} we mean an open subscheme of $\Spec \Oo_K$ for $K$ a number field. Given a base scheme $Z$, we say $X \to Z$ is \emph{smoothly open} if it admits a smooth proper compactification with relative strict normal-crossings boundary having smooth proper strata. Although some results apply more broadly, throughout the article, $Z$ may be taken to be an open integer scheme and $X$ a smoothly open $Z$-scheme. 

To distinguish Eilenberg MacLane spaces from Quillen K-groups, we denote the former by e.g. `$\Kk(\pi,i)$'. More generally, if $\pi: \Ee\op \to \mathbf{Grp}$ is a presheaf of groups, we signal by $\Kk(\pi,i)$ a presheaf of homotopy types with values
\[
\Kk(\pi,i)(T) = \Kk\big(\pi(T),i\big).
\]

Given a module $E$ over a ring $R$, we define the \emph{covariant total space} by 
\[
\VV^\lor E  = \Spec \Sym E^\lor.
\]

\section{Criteria for representability in an abstract setting}
\label{Sec:Rep}

\subsection{Augmentation presheaves and the Iwanari tower}
\label{sSec:AugPre}

We begin by defining the sort of functors whose representability we are interested in.

\sspar{Definition of augmentation spaces}
\label{BB33}
Let $\Dd \in \CAlg (\PrLst)_{\Dd(F)/}$ be a presentably symmetric monoidal stable $F$-linear \oo-category over a field $F$ of characteristic 0 and let $\Cc = \CAlg(\Dd)$. (In our applications, $F$ will be either $\QQ$ or $\Qp$.) Let $\Dd(F)$ be the derived \oo-category of $F$-\textit{vector spaces}. Let $\Cc(F) = \CAlg \Dd(F)$. We write $F(0)$ for the tensor units of $\Dd$, $\Dd(F)$ and we write $\one$ for the (essentially unique) trivial commutative algebra with underlying object $F(0)$. If $C \in \Cc$ is a commutative algebra, then there's an evident associated functor 
\[
\widetilde{\Aaug}(C): 
\Cc(F) \to \HoT
\]
such that
\[
\widetilde{\Aaug}(C)(R) = 
\Hom_\Cc(C, R \otimes \one). 
\]
Throughout this article we will focus on the restriction  $\Aaug(C)$ of $\widetilde {\Aaug}(C)$ to the category of discrete $F$-algebras. We think of $\Aaug(C)$ as a presheaf of homotopy types on the category $\Aff_F$ of affine $F$-schemes. Our main object of interest in this article is the associated presheaf of sets $\pi_0 \Aaug(C)$. As a matter of notation, we set
\[
\Aug(C) := \pi_0 \Aaug(C);
\]
however, we will frequently write out `$ \pi_0 \Aaug(C)$' for emphasis. 

\begin{ssDefinition}
\label{BE44}
Given $A \in \Cc = \CAlg \Dd$, $V \in \Dd$, and $f: V \to A$ a morphism in $\Dd$, we define \emph{the coherent Hirsch extension $A'$ of $A$ by $V$} to be the pushout in $\Cc$ 
\[
\xymatrix{
A' & \one \ar[l]
\\
A \ar[u] & \Sym V \ar[l]^{\bar f} \ar[u]_-\ep
}
\]
of the induced map $\bar f$ along the canonical augmentation $\ep$. 
\end{ssDefinition}

\sspar{Definition of Iwanari tower}
\label{BB34}
Fix a commutative algebra $C \in \Cc$ as above. Following Iwanari \cite{Iwanari}, we define the \emph{coherent Sullivan tower} or \emph{Iwanari tower} of $\Cc$ below left 
\[
\xymatrix
@ R = 4ex
{
C 
&& V_n \ar[r] & \fkS_n C \ar[r] & C \ar[r] & V_n[1]
\\
\vdots \ar[u] && \fkS_{n+1}C  & \one \ar[l] 
\\ 
\fkS_2 C \ar[u] & \Sym V_2 \ar[l] & \fkS_n C \ar[u] & \Sym V_n \ar[l] \ar[u].
\\
\fkS_1 C \ar[u] & \Sym V_1 \ar[l] 
\\
\one \ar[u] & \Sym V_0 \ar[l]
}
\]
via the exact triangles in $\Dd$ and the coherent Hirsch extensions in $\Cc$ to the right. We set 
\[
\Aaug^n(C) := \Aaug(\fkS_n C).
\]

\sspar{Definition of dual tower and associated exact couple}
\label{BB35}
Abbreviating $\Aa^nC:= \Aaug^n(C)$ and
\[
\Vv^n = \Aaug(\Sym V_n) 
\]
so that
\[
\Vv^n(R) = \Hom_\Dd \big( V_n, R(0) \big),
\]
we obtain a tower of presheaves of homotopy types (below left)
\[
\xymatrix{
\Aa \ar[d]
\\
\vdots \ar[d] &
&&
\vdots \ar[d]
\\
\Aa^2 \ar[d] \ar[r] & \Vv^2
&& 
\pi_* \Aa^2 \ar[d] \ar[r] & \pi_* \Vv^2 \ar[ul]|{(-1)}
\\
\Aa^1 \ar[d] \ar[r] & \Vv^1
&&
\pi_*\Aa^1 \ar[d] \ar[r] & \pi_* \Vv^1 \ar[ul]|{(-1)}
\\
\{\ast\} \ar[r] & \Vv^0.
&&
\pi_*(\ast) \ar[r] & \pi_*\Vv^0. \ar[ul]|{(-1)}
}
\]
For suitable choice of base-points, we obtain an exact couple of homotopy groups / sets (to the right). We find it psychologically helpful to unfold the latter into a grid of knight's moves:
\begin{footnotesize}
\[
\xymatrix
@ C = 3ex
{
\pi_3 \Aa^1 \ar@{~>}[r]^-{\al_3} &
\pi_3 \Vv^1 \ar@{~>}[r] &
\pi_2 \Aa^2 \ar[r]^-{\al_2} \ar@{~>}[d] &
\pi_2\Vv^2 \ar[r] & \pi_1 \Aa^3 \ar[r]^-{\al_1} \ar[d] & \pi_1 \Vv^3 \ar[r] & \pi_0 \Aa^4 \ar[r] \ar[d] & \pi_0 \Vv^4
\\
&&
\pi_2 \Aa^1 \ar@{~>}[r] 
&
\pi_2\Vv^1 \ar@{~>}[r] & \pi_1 \Aa^2 \ar[r] \ar@{~>}[d] & \pi_1 \Vv^2 \ar[r] & \pi_0 \Aa^3 \ar[r] \ar[d] & \pi_0 \Vv^3
\\
&&&& 
\pi_1 \Aa^1 \ar@{~>}[r] & \pi_1 \Vv^1 \ar@{~>}[r] & \pi_0 \Aa^2 \ar[r] \ar@{~>}[d] & \pi_0 \Vv^2
\\
&&&&&& 
\pi_0 \Aa^1 \ar@{~>}[r] \ar[d] & \pi_0 \Vv^1
\\
&&&&&& \ast &.
}
\]
\end{footnotesize}
\noindent The squiggly arrows highlight a long exact sequence. 

\begin{ssRemark}
\label{BB36}
The $1$st Iwanari truncation $\fkS_1C$ of a commutative algebra $C \in \Cc$ often has a simple structure. Suppose the underlying object $C \in \Dd$ decomposes as
\[
C \simeq F(0) \oplus \widetilde C
\]
so that the unit map $\one \to C$ corresponds to the coprojection. Then
\[
\fkS_1C \simeq \Sym \widetilde C.
\]
Indeed, the resulting exact triangle
\[
\widetilde C[-1] \to F(0) \to C \to \widetilde C
\]
shows that $V_0 \simeq \widetilde C[-1]$. Since the functor $\Sym$ commutes with colimits, we then find that 
\begin{align*}
\fkS_1C  &
\simeq \one \otimes_{\Sym \widetilde C[-1]} \one
\\
& \simeq \Sym \left( 
0 \amalg_{\widetilde C[-1]} 0
\right)
\\
& \simeq \Sym \widetilde C.
\end{align*}
\end{ssRemark}

\subsection{Coarse criterion}
\label{sSec:Coarse}

As mentioned in the introduction, our first criterion has the aesthetic advantage of being worded entirely in terms of the (``abelian'') objects $V_i \in \Dd$ and the associated rational infinite loop spaces $\Vv^i(F) = \Hom_\Dd \big(V_i, F(0) \big)$. It has the disadvantage that the vanishing of $\pi_2$ often fails in practice. 

\begin{ssProposition}
\label{BB37}
Let $C \in \Cc = \CAlg \Dd$ be a commutative algebra in a presentably symmetric monoidal stable $F$-linear \oo-category as in paragraph \ref{BB33}. Consider the presheaves of homotopy types $\Aa^i = \Aaug^i(C)$, $\Vv^n = \Aaug(\Sym V_n)$ associated to the Iwanari tower of $C$ (paragraphs \ref{BB34}, \ref{BB35}). Fix $n \ge 1$ and assume 
\begin{itemize}
\item[(a)]
$V_i$ is a sum of compact objects for all $0 \le i \le n-1$,
\item[(b)]
$\pi_2\Vv^i(F) = 0$ for all $0 \le i \le n-2$, and
\item[(c)]
$\pi_1\Vv^i(F)$, $\pi_0\Vv^i(F)$
are finite dimensional for all $0 \le i \le n-1$.
\end{itemize}
Then 
\begin{enumerate}
\item
$\pi_0\Aaug^n(C)$ is representable by a finite-type affine $F$-scheme,
\item
for every $F$-algebra $R$, every connected component of $\Aaug^{n-1}(C)(R)$ is simply connected, and
\item
we have
\[
\dim \pi_0\Aaug^n(C) 
\le 
\sum_{i=0}^{n-1} \dim_F \pi_1 \Vv^i(F),
\]
with equality if $\pi_0 \Vv^i(F) = 0$ for $i = 1, \dots, n-1$.
\end{enumerate}
\end{ssProposition}

The proof spans paragraphs \ref{BB38} -- \ref{D06}.

\sspar{}
\label{BB38}
Let $E \xto{g} B$ be a fibration of topological spaces, fix $b \in B$ and let $\Ff = g\inv(b)$. Assume every path component of $E$ is simply connected. Then $\pi_1(B,b)$ acts freely and transitively on the fibers of $\pi_0\Ff \to \pi_0E$.

\begin{ssLemma}
\label{D03}
Let $M$ and $N$ be objects of $\Dd$, $E$ an $F$-\textit{vector space} and $i \in \NN$ a natural number. Assume $M$ is a sum of compact objects and assume the $F$-\textit{vector space} $\pi_i \Hom_\Dd(M,N)$ is finite dimensional. Then
\[
\pi_i\Map_\Dd(M, N\otimes E)
 = (\pi_i\Map_\Dd(M,N)) \otimes_F E.
\]
\end{ssLemma}

\begin{proof}
This reduces to a statement about the $F$-linear triangulated category $h\Dd$ in the following way. Write $M$ as a sum of compact objects $\bigoplus_{\al \in A} M_\al$. Then 
\begin{align*}
\pi_i \Hom_\Dd(M,N)
 &= \Hom_{h\Dd}(M[i],N) 
\\
&= \Hom_{h\Dd}(\bigoplus_{\al \in A} M_\al[i], N)
\\
&= \prod_{\al \in A} \Hom_{h\Dd}(M_\al[i], N).
\end{align*}
Since this space is assumed to be finite dimensional, the product is actually a finite sum
\[
\bigoplus_{\al \in A^0} \Hom_{h\Dd}(M_\al[i], N)
\]
indexed by a finite subset $A^0 \subset A$. Let $\Ee$ be a basis for $E$. It follows that
\begin{align*}
\Map_{h\Dd}(M[i],N) \otimes_F E
&= \bigoplus_{\al \in A^0} \Hom_{h\Dd}(M_\al[i],N)^{\bigoplus \Ee}
\\
&= \bigoplus_{\al \in A^0} \Map_{h\Dd}(M_\al[i],N^{\bigoplus \Ee})
\\
&= \Map_{h\Dd}(M[i], N \otimes E).
\end{align*}
Hence,
\begin{align*}
\pi_i\Map_\Dd(M, N\otimes E)
&= \pi_0 \Map_\Dd(M[i], N\otimes E)
\\
&= \Map_{h\Dd}(M[i], N\otimes E)
\\
&= \Map_{h\Dd}(M[i], N) \otimes_F E
\\
& = (\pi_i\Map_\Dd(M,N)) \otimes_F E.
\qedhere
\end{align*}
\end{proof}

\begin{ssLemma}
\label{rep4}
Let $V \in \Dd$. Assume $V$ is a sum of compact objects and assume $\pi_i \Hom_\Dd \big(V, F(0) \big)$ is finite dimensional as an $F$-vector space. Then the functor 
\[
\Vv: \opnm{Aff}_F\op \to \Set
\]
defined on an $F$-algebra $R$ by 
\[
R \mapsto 
\pi_i\Hom_\Dd \big( V, R \otimes F(0)\big)
\]
is representable by the vector group 
\[
\Vv = \Spec \Sym \pi_i \Hom_\Dd \big( V, F(0)\big)^\lor.
\]
\end{ssLemma}

\begin{proof}
For any $F$-algebra $R$, we have bijections
\begin{align*}
& \Hom_{\Vect(F)} \Bigg(
\Big(\pi_i\Hom_\Dd \big(V, F(0) \big)\Big)^\lor, R 
\Bigg)
\\
&= \Big(\pi_i\Hom_\Dd \big(V, F(0) \big)\Big) \otimes R
\\
&= \pi_i\Hom_\Dd \big(V, R(0) \big)
\intertext{(by Lemma \ref{D03})}
&= \Vv(R)
\end{align*}
natural in $R$. 
\end{proof}

\sspar{}
\label{D04}
Conditions (b) and (c) of proposition \ref{BB37} posit that for various $i$, $j$,
$
\pi_j\Vv^i(F)
$
is finite dimensional. Applying Lemma \ref{rep4}, we find that for such $i$, $j$, $\pi_j \Vv^i$ is represented by the vector group $\Spec \Sym \big( \pi_j\Vv^i(F) \big)^\lor$.

\sspar{}
\label{D05}
For the base case ($n=1$), we note that for every $F$-algebra $R$,
\[
\Aa^0(R) = \Map_\Cc(\one, \one \otimes R)
\]
is contractible, and the long exact sequence of homotopy groups associated to the fiber sequence
\[
\Aa^1 \to \ast \to \Vv^0
\]
provides an isomorphism of presheaves of sets
\[
\pi_1 \Vv^0 \xto{\sim} \pi_0 \Aa^1. 
\]
Thus, in particular, $\Aa^0$ is simply connected and $\pi_0\Aa^1$ is representable by an affine space of the stated dimension.

\sspar{}
\label{D06}
Now fix arbitrary $n \in \NN$. For any $F$-algebra $R$ and any $m \in \NN$, we consider the fiber sequence of pointed homotopy types
\[
\Aa^m(R) \xto{f^m} \Aa^{m-1}(R) \xto{g^m} \Vv^{m-1}(R).
\]
We may identify $g^{n-1}$ with a fibration of topological spaces with fiber $f^{n-1}$. Fix arbitrarily $x \in \Aa^{n-1}(R)$ with image $ \in \Aa^{n-2}(R)$ denoted again by $x$; its image in  $\Vv^{n-2}(R)$ lies in the connected component of the base-point $0$. The associated long exact sequence includes the portion 
\[
\pi_2\big(\Vv^{n-2}(R), 0\big) 
\to 
\pi_1\big(\Aa^{n-1}(R), x\big) 
\to 
\pi_1\big(\Aa^{n-2}(R), x\big).
\]
We've assumed that the first group vanishes and we may assume for an induction on $n$ that the last group vanishes, so that $\pi_1\big(\Aa^{n-1}(R), x\big) = 0$ as well.

Under our inductive assumptions, the map
\[
\pi_0\Aa^{n-1} 
\xto{\pi_0g^n} 
\pi_0 \Vv^{n-1}
\]
(being natural in $R$) is a map of finite type affine schemes, so its fiber $Z^{n-1}$ is again a finite-type affine scheme. As part of the long exact sequence of homotopy groups, $f^n$ induces a surjection
\[
\pi_0 \Aa^n(R) \xto{h^n} Z^{n-1}(R).
\]
Moreover, $\pi_1\Vv^{n-1}(R)$ acts freely and transitively on each fiber of $h^n$. Indeed, we may identify $g^n$ with a fibration of topological spaces with fiber $f^n$, in which case this is the general fact recalled in paragraph \ref{BB38}. This completes the induction and establishes Proposition \ref{BB37}. \qed

\subsection{Refined criterion}
\label{sSec:Refined}

We now formulate a criterion for representability that is harder to apply, but applies more broadly.

\begin{ssDefinition}[Presheaf of relative homotopy groups]
\label{BB40}
Let $\Ee$ be an \oo-category, let $f: \Xx \to \Tt$ be a morphism of presheaves of homotopy types, and let $x: \Xx \from \Tt$ be a quasi-section (i.e. $f \circ x \sim \m{id}_\Tt$). Translating $\Xx \to \Tt$ via the equivalence
\[
\tag{*}
\PSh(\Ee_{/\Tt}) \simeq \PSh(\Ee)_{/\Tt},
\]
we obtain a presheaf $\Xx_{/\Tt}$ on the slice category (a ``sliced presheaf'') equipped with a global base point. The associated presheaf of groups $\pi_i(\Xx_{/\Tt}, x)$ is equally a grouplike algebra in sliced presheaves. Translating back via (*) (while noting that (*) is canonically monoidal for the Cartesian monoidal structures), we obtain a grouplike algebra in presheaves over $\Tt$
\[
\pi_i (\Xx / \Tt, x) \to \Tt,
\]
the \emph{presheaf of relative homotopy groups of $\Xx$ over $\Tt$ along $x$}.
\end{ssDefinition}

\begin{ssRemark}
In Definition \ref{BB40}, if $\Tt$ is a presheaf of discrete homotopy types, then $\pi_i(\Xx/\Tt, x)$ is just the presheaf of sets 
\[
S \mapsto
\set{(f, \ga)}
{f: S \to \Tt, \ga \in \pi_i \big(\Xx(S), x \circ f \big)},
\]
which is evidently functorial with respect to morphisms in the homotopy category $h\Ee\op$. In our applications in this article, $\Ee$ will always be a 1-category and $\Tt$ will be representable, hence discrete. 
\end{ssRemark}

\begin{ssRemark}
\label{BB41}
We'll often apply the discussion of paragraph \ref{BB40} to a situation given by a presheaf of homotopy types $\Aa: \Ff\op \to \HoT$ on an \oo-category $\Ff$ (actually a 1-category), an object $L \in \Ff$, and a morphism $g$ of presheaves: 
\[
\xymatrix{
& \Aa \ar[d] 
\\
L \ar[r]_-g \ar@{.>}[ur]^-x & \pi_0 \Aa.
}
\]
Since the map of homotopy types $\Aa(L) \to \pi_0\Aa(L)$ admits a section, the Yoneda lemma provides an $L$-valued base point of $\Aa$, as shown. We then consider the base-change $\Aa_L = L \times_{\pi_0\Aa} \Aa$ of $\Aa$ along $g$ viewed as a presheaf over $L$ equipped with section induced by $x$. Definition \ref{BB40} provides an associated presheaf of relative homotopy groups 
\[
\pi_i(\Aa_L / L, x) \to L.
\]
\end{ssRemark}

\begin{ssProposition}
\label{BB44}
Let $\Ee$ be an \oo-category, suppose given a homotopy pullback of presheaves of homotopy types
\[
\xymatrix{
\Aa' \ar[d]_-{\phi} \ar[r]
&
\ast \ar[d]^{0}
\\
\Aa \ar[r]_-{\psi}
&
\Vv,
}
\]
and let $L$ be the fiber of $\pi_0 \Aa \to \pi_0 \Vv$ over $0$. Assume $\pi_1(\Vv,0)(R)=\pi_1\big(\Vv(R),0\big)$ is abelian for all $R$.  Then the map
\[
\pi_0 \phi: \pi_0 \Aa' \to L \subset \pi_0 \Aa
\]
is a pseudotorsor under a presheaf of abelian groups $Q \to L$ over $L$.
\end{ssProposition}

\begin{ssRemark}
Below, we apply Proposition \ref{BB44} only to the case $\Ee = \Aff_k$. The added generality afforded by allowing $\Ee$ to be an \oo-category makes no difference in the proof, since presheaves of homotopy groups factor through the homotopy category $h\Ee$. This added generality may come in handy if we wish to eventually replace $\Aff_k$ by $\Cc(k)\op = \big(\CAlg \Dd(k)\big)\op$.
\end{ssRemark}

\begin{proof}
We view $\Vv$ as pointed and drop the base-point from our notation for homotopy groups. There's an action of $\pi_1 \Vv$ on $\pi_0 \Aa'$ which is transitive on nonempty fibers; we may equally view it as an action of $\pi_0 \Aa \times \pi_1 \Vv$ over $\pi_0 \Aa$. Moreover, given $T \in \Ee$ and $x: \ast \to \Aa'(T)$ a point, the stabilizer of the connected component $[x]$ of $x$ in $\Aa'(T)$ is equal to the image of 
\[
\pi_1(\psi ,\phi x): \pi_1 \big( \Aa(T), \phi x \big) \to \pi_1 \Vv(T).
\] 
Since $\pi_1 \Vv(T)$ is abelian, the image
\[
\Im \big( \pi_1 [\phi x] \big)  = \Im \pi_1(\psi, \phi x)
\]
of $\pi_1(\psi, \phi x)$ depends only on the connected component $[\phi x]$ of
\[
\ast \to \Aa'(T) \to \Aa(T).
\]
Thus, we may unambiguously define $Q(T)$ to be the set of pairs
\[
Q(T) = 
\set{\big([y], [\ga]\big)}
{ [y] \in L(T), 
[\ga] \in \pi_1 \Vv(T) / \pi_1[y]}.
\]
By construction, $Q(T)$ acts freely and transitively on the fibers of $\pi_0 \phi(T)$. 

We now show that the quotient map
\[
L (T) \times \pi_1 \Vv(T) 
\surj Q(T)
\]
of abelian groups over $L(T)$ is natural in $T$. Fix $g: T' \to T$ a morphism in $\Ee$. The morphism of presheaves $\phi$ includes the data of a coherently commuting square of homotopy types (below right)
\[
\xymatrix{
[y|_{T'}] \ar@{}[r]|-\subset & 
\Aa(T') \ar[r] & \Vv(T')
\\
[y] \ar@{}[r]|-\subset \ar[u] & 
\Aa(T) \ar[u]_-{\rho_\Aa} \ar[r] & \Vv(T). \ar[u]_-{\rho_\Vv}
}
\]
We may add to this a choice of compatible connected components, as shown on the left. It follows that the restriction map $\rho_\Vv$ maps $\Im \pi_1[y]$ to $\Im \pi_1 [y|_{T'}]$, hence that $\rho_\Vv$ induces a restriction map $\rho_Q$ for $Q$. We obtain a commuting square
\[
\xymatrix{
L(T') \times \pi_1\Vv(T') \ar[r] & Q(T')
\\
L(T) \times \pi_1\Vv(T) \ar[u] \ar[r] & 
Q(T) \ar[u]_-{\rho_Q},
}
\]
hence the naturality. It also follows that $Q$ is functorial. 
\end{proof}

\begin{ssProposition}
\label{MM55}
In the situation and the notation of paragraph \ref{BB33} ($\Dd \in  \CAlg (\PrLst)_{\Dd(F)/}$, $\Cc = \CAlg \Dd$, ...) suppose given a coherent Hirsch extension $\Cc$ as shown below left
\[
\tag{*}
\xymatrix{
A'   & \one \ar[l]
&&
\Aa' \ar[d]^-{\phi} \ar[r] & \{\ast\} \ar[d] 
\\
A \ar[u] & \Sym V. \ar[l] \ar[u]
&&
\Aa \ar[r]_-\psi \ar[d] & \Vv \ar[d]
\\
&
&
& \pi_0 \Aa \ar[r]^-{\pi_0\psi} & \pi_0 \Vv
\\
&&
L \ar@{.>}[uuur]^-{x'} \ar[ur]_\iota \ar[r] &
\ast \ar[ur]_0
}
\]
We denote associated presheaves of homotopy types by $\Aa'$, $\Aa$, $\Vv$ as shown to the right. Assume the fiber $L$ of $\pi_0\psi$ over the canonical base-point $0$ is representable by an affine scheme. Then $Q$ may be described as the cokernel of a map of presheaves of relative fundamental groups induced by a map of presheaves of homotopy types over $L$. Moreover, the induced map
\[
\pi_0 \phi: \pi_0 \Aa' \to L \subset \pi_0 \Aa
\]
makes $\pi_0 \Aa'$ into a trivial $Q$-torsor over $L$.
\end{ssProposition}

\begin{proof}
By the long exact sequence in homotopy, there exists a point $x': \ast \to \Aa'(L)$ lifting $\iota$. By the Yoneda lemma, $x'$ gives rise to an $L$-valued base-point of $\Aa'$ lifting $\iota$, as indicated in the diagram. 

Let $x$ be the image of $x'$ in $\Aa$. Let $\Aa_L$ be the base-change of $\Aa$ along $L \to \pi_0\Aa$, let $\Vv_0$ be the base-change of $\Vv$ along the zero-map $\ast \to \pi_0\Vv$, and let $(\Vv_0)_L \simeq L \times \Vv_0$ be the further base-change of $\Vv_0$ along $L \to \ast$. Since the square
\[
\xymatrix{
& \Aa \ar[r] & \Vv
\\
L \ar[ur]^-{x} \ar[r] & \ast \ar[ur]_-0
}
\]
commutes up to homotopy, there's an induced map of presheaves of relative fundamental groups over $L$
\[
\pi_1( \Aa_L / L , x)
\to
\pi_1\big( (\Vv_0)_L / L, id_L \times 0 \big)
\simeq
L \times \pi_1(\Vv, 0).
\]
The presheaf cokernel is equal to $Q$.

Finally, the existence of a section of $\Aa'_L \to L$ induced by $x'$ shows that $\pi_0\Aa' \to L$ is a trivial torsor (and not merely a pseudotorsor). 
\end{proof}

\begin{ssCorollary}
\label{BB45}
In the situation and the notation of paragraph \ref{BB35} given by $\Dd \in \CAlg (\PrLst)_{\Dd(F)/}$, $\Cc = \CAlg \Dd$, and $C \in \Cc$ with Iwanari tower $\big\{\fkS_n C, V_n \big\}_n$, and dual tower $\big\{\Aa^n, \Vv^n\big\}_n$, let
\[
L_n = \pi_0\Aa^n \times_{\pi_0 \Vv^n} \{0\}
\]
and let $Q_n \to L_n$ be the presheaf of abelian groups over $L$ constructed in Proposition \ref{BB44}. Fix $n$ and assume that for every $0 \le i \le n$
\begin{itemize}
\item[(a)]
$V_i$ is a sum of compact objects,
\item[(b)]
$\pi_0 \Vv^i(F)$ is finite dimensional,
\item[(c)]
$Q_i$ is representable.  
\end{itemize}
Then $\pi_0 \Aa^{n+1}$ is representable. 
\end{ssCorollary}

We add to our refined criterion (Corollary \ref{BB45}) a remark (\ref{VF44}) and a proposition (\ref{KK61}) of possible relevance for establishing the criterion in practice.  

\begin{ssRemark}
\label{VF44}
In the situation and the notation of Proposition \ref{MM55}, the sequence 
\[
\xymatrix{
\pi_1(\Aa_L/L,x) \ar[r]^-\al &
L \times \pi_1(\Vv,0) \ar[r] &
\pi_0 \Aa' \ar@{->>}[d]
\\
&& L
}
\]
(in which $\pi_0 \Aa'$ is a trivial torsor over $L$ under $Q = \cok \al$) may be extended to a long exact sequence of presheaves of relative homotopy groups, as we now explain. Diagram \ref{MM55}(Right) gives rise to a coherently commuting diagram (below left)
\[
\tag{*}
\xymatrix{
L \times_{\pi_0 \Aa} \Aa' \ar[r] \ar[d]
&
\ast \ar[d]
&&
L \times_{\pi_0 \Aa} \Aa' \ar[r] \ar[d] &
L \ar[d]
\\
L \times_{\pi_0\Aa} \Aa \ar[r] \ar[d] &
\ast \times_{\pi_0 \Vv}  \Vv  \ar[d]
&&
L \times_{\pi_0 \Aa} \Aa \ar[r] \ar[dr] & 
L \times_{\pi_0 \Vv} \Vv \ar[d]
\\
L \ar[r] & \ast ,
&&
& L . 
}
\]
from which a coherently commuting square in the slice category $\PSh(\Aff_F)_{/L}$, as shown on the right. As a matter of notation, we'll indicate the base changes appearing above right with subscripts ($\Aa'_L$, $\Aa_L$, $\Vv_L$). 

Fix an affine $L$-scheme $f: T \to L$ and let $g$ be the composition 
\[
\xymatrix{
T \ar[rr]^-g \ar[dr]_-f && \pi_0 \Aa
\\
& L \ar[ur]_-h &.
}
\]
Evaluating the sequence (*)$_\text{left}$ at $f:T \to L$, we obtain the sequence
\[
\xymatrix{
\{g\} \times_{\pi_0\Aa(T)} \Aa'(T)
\ar[d]
\\
\{g\} \times_{\pi_0\Aa(T)} \Aa(T)
\ar[r]
&
\{0 \} \times_{\pi_0\Vv(T)} \Vv(T),
}
\]
which is just the result of expunging connected components\footnote{We're free to expunge connected components in the computation of a homotopy fiber in the following sense. Suppose given a pullback square in $\HoT$ as in the right portion of the diagram below
\[
\tag{**}
\xymatrix{
X \ar[r] \ar[d] & X \ar[r] \coprod X' \ar[d]_-{\al \coprod \be} \ar[r] & \ast \ar[d]^-z
\\
Y \ar[r] & Y \ar[r] \coprod Y' \ar[r]_\ga & Z \coprod Z'
}
\]
where $\ga$, $z$ land in $Z$. Then the outer square is again a pullback square, and remains so if we get rid of $Z'$. } from the fiber sequence 
\[
\xymatrix{
\Aa'(T) \ar[d]
\\
\Aa(T) \ar[r]
& \Vv(T).
}
\] 
Thus, the square in diagram (*)(upper right) forms a fiber sequence in $\PSh(\Aff_F)_{/L} \simeq \PSh(\Aff_L)$. Additionally, if $y$ denotes the section of $\Vv_L \to L$ induced by the image of $x$ in $\Vv$, we have for $i\ge 1$,
\[
\pi_i (\Vv_L/L, y) \simeq \pi_i(\Vv_L/L, 0) \simeq L \times \pi_i(\Vv,0),
\]
a canonical isomorphism since $\pi_1$ is abelian. 

In this way we obtain a long exact sequence of presheaves of relative homotopy groups 
\[
\xymatrix{
\pi_*( \Aa'_L / L, x') \ar[d]
\\
\pi_* ( \Aa_L/L, x) \ar[r]
&
L \times \pi_* (\Vv,0) \ar[ul]|-{(-1)}
}
\]
over $L$, in which 
\[
\pi_0( \Aa'_L/L,x') \surj
\pi_0( \Aa_L/L,x) \simeq L
\]
is a trivial torsor under 
\[
Q := \cok
\Big(
\pi_1(\Aa_L/L, x) \to L \times \pi_1(\Vv,0)
\Big).
\]
\end{ssRemark}

\begin{ssProposition}
\label{KK61}

Situation and notation as in paragraph \ref{BB35} given by $\Dd \in \CAlg (\PrLst)_{\Dd(F)/}$, $\Cc = \CAlg \Dd$, and $C \in \Cc$ with Iwanari tower $\big\{\fkS_n C, V_n \big\}_n$, and dual tower $\big\{\Aa^n, \Vv^n\big\}_n$. Fix $n$. Assume that for every $0 \le i \le n$, $V_i$ is a sum of compact objects. Assume for all nonnegative $i,j$ such that $0 \le i+j \le n$, $\pi_i\Vv^j(F)$ is finite dimensional. Then for any global section $x: \ast \to \Aa^n$, the presheaves of groups $\pi_i \big( \Aa^j, x \big)$ ($i+j \le n$, $i \ge 1$) are representable by unipotent groups. 
\end{ssProposition}

\begin{proof}
The base point $x$ endows $\Aa^{j}$ ($j \le n$) with a base point, and we'll use these base points for homotopy groups. We may assume for an induction on $i+j$ that the groups $\pi_i(\Aa^j)$ for $i+j < n$ are unipotent. Since our finiteness assumptions guarantee that the relevant homotopy groups of $\Aa^1$ are representable by vector groups, we may further assume for an induction on $j$ that the homotopy presheaves $\pi_i \big( \Aa^{j'})$ with $i+j' = n$,$j'<j $ are representable by unipotent groups. The exact sequence of presheaves of groups
\[
\xymatrix{
\pi_{i+1} \Aa^{j-1} \ar[r]^-\be &
\pi_{i+1} \Vv^{j-1} \ar[r] &
\pi_{i} \Aa^{j} \ar[d]
\\
&&
\pi_{i} \Aa^{j-1} \ar[r]^-\al &
\pi_{i} \Vv^{j-1}
}
\]
shows first that $\pi_{i} \Aa^{j}$ is a trivial torsor over a unipotent group (the kernel of $\al$) under a unipotent group (the presheaf cokernel of $\be$), hence is representable by a group scheme; it then also shows that this group scheme is an extension of a unipotent group by a unipotent group, hence itself unipotent. 
\end{proof}

\section{Representability in Examples}
\label{Sec:RepEx}

\subsection{Highly structured algebras in motives and realizations}
\label{sSec:Rev}

\sspar{}
\label{EA23}
Let $Z$ be a Noetherian separated scheme. We let 
\[
\DM(Z, \QQ)^\otimes \in \CAlg(\m{Pr}^L_\m{st})_{\Dd(\QQ)/}
\]
be the symmetric monoidal  $\QQ$-linear stable \oo-category of mixed motivic complexes over $Z$ with $\QQ$-coefficients. This is the underlying symmetric monoidal \oo-category of the symmetric monoidal model categories considered by Ayoub \cite{AyoubSixI, AyoubSixII} and by Cisinski-D{\'e}glise \cite{CisDeg}. We recall that as a matter of notation we allow ourselves to identify a symmetric monoidal \oo-category with the underlying \oo-category by dropping the superscript $\otimes$. We let
\[
\DM^\m{il}(Z, \QQ) \subset \DM(Z, \QQ)
\]
be the presentably symmetric monoidal full subcategory spanned by Ind-lisse motives. When the base $Z$ is fixed throughout a discussion, we abbreviate 
\[
\Dd_M = \DM^\m{il}(Z, \QQ).
\]

\sspar{}
\label{EA24}
Let $\Sm^o_Z$ be the category of smoothly open $Z$-schemes (\ref{notation}). The \oo-category $\Dd_M$ comes equipped with a symmetric monoidal functor
\[
C_*: \Sm^o_Z \to \Dd_M,
\]
the \emph{homological motives functor}. Composition with the dualization $M \mapsto M^\lor = \Hhom\big(M, \QQ(0) \big)$ gives rise to the symmetric monoidal \emph{cohomological} motives functor 
\[
C^*: (\Sm_Z^o)\op \to \Dd_M.
\]
The equivalence $\CAlg \big((\Sm_Z^o)\op\big) \simeq (\Sm_Z^o)\op$ means that the latter upgrades canonically to a functor 
\[
C^*: (\Sm_Z^o)\op \to \Cc_M := \CAlg \Dd_M,
\]
the \emph{algebraic motives functor}. See Dan-Cohen--Schlank \cite{rmps_arXiv} for a model categorical approach to this construction, and Iwanari \cite{Iwanari} for a direct infinity-categorical treatment. 

\begin{ssDefinition}
\label{EA25}
Let $X \in \Sm^o_Z$ be a smoothly open $Z$-scheme. We define the \emph{motivic cochain algebra of $X$} to be the highly structured commutative algebra $C^*(X) \in \Cc_M$. We define the \emph{$n$th truncated motivic augmentation space of $X$ by}
\[
\Aaug^n(X) = \Aaug \big(\fkS_n C^*(X)\big).
\]
\end{ssDefinition}

\subsubsection{}
\label{EA26}
Now let $K_\pP$ be a finite extension of $\Qp$, let $\Oo_\pP$ be its ring of integers, and let $Z_\pP = \Spec \Oo_\pP$. Let $\MFphi(K_\pP)$ be the $\Qp$-Tannakian category of admissible filtered $\phi$ modules over $K_\pP$. Let $\Dd\big(\Ind \MFphi(K_\pP)\big)$ be the symmetric monoidal derived \oo-category of its Ind-completion. When the base $Z_\pP$ is clear from the context, we'll abbreviate
\[
\Dd_\Fphi = \Dd\big(\Ind \MFphi(K_\pP)\big)
\]
Work of Déglise-Niziol \cite{DegliseNiziol} and others provides a symmetric monoidal \oo-categorically enhanced realization functor
\[
\Re_\Fphi: \DM^\m{il}(Z_\pP, \QQ) \to 
\Dd_\Fphi,
\]
hence also a functor 
\[
\Re_\Fphi: \CAlg
\DM^\m{il}(Z_\pP, \QQ) \to 
\Cc_\Fphi := \CAlg \Dd_\Fphi.
\]
Composing with the cohomological motives functor, we obtain a functor
\[
C^*_\Fphi: (\Sm^o_{Z_\pP})\op  \to 
\Cc_\Fphi.
\]

\begin{ssDefinition}
\label{EA27}
Let $X \to Z_\pP$ be a smoothly open variety. We define the \emph{filtered $\phi$ cochain algebra of $X$} to be the filtered $\phi$ realization 
\[
C^*_\Fphi(X) = \Re_\Fphi C^*(X)
\]
of the motivic cochain algebra $C^*(X)$. We define the \emph{$n$th truncated filtered $\phi$ augmentation space of $X$ by}
\[
\Aaug^n_\Fphi(X) = \Aaug \big(\fkS_n C_\Fphi^*(X)\big).
\]

\end{ssDefinition}

\subsection{$\Gm$-bundles over semiabelian schemes}
\label{sSec:GmBun}

\subsubsection{Semiabelian schemes}
\label{TC11}
Let $Z$ be a Noetherian scheme and let $G \to Z$ be a semiabelian $Z$-scheme. By Ancona et al. \cite{AnconaHuberII} and Iwanari \cite{Iwanari},
\[
C^*(G) \simeq \Sym \big(H^1[-1]\big)
\]
is the free commutative algebra on a compact object $H^1[-1] \in \Dd_M$. Hence, whenever 
\[
\tag{*}
\pi_0 \Hom_{\Dd_M} \big( H^1 [-1], \QQ(0) \big) 
\simeq G(Z) \otimes \QQ
\]
is finite dimensional, Lemma \ref{rep4} applies to show that 
\[
\Aug \big( C^*(G) \big)
\simeq
\VV^\lor G(Z) \otimes \QQ. 
\]

\begin{ssTheorem}
\label{TC12}
Let $L^\star \to G \to Z$ be a $\Gm$-bundle over a semiabelian scheme over a Noetherian scheme $Z$ such that $G(Z)$ and $\Gm(Z)$ have finite rank. Assume $Z$ satisfies
\[
\tag{*}
\pi_0\Hom_{\Dd_M(Z)} \big( \QQ(0), \QQ(1)[2] \big) = 0.
\]
Then the presheaf of sets $\Aug \big( C^*(L^\star) \big)$ is representable by a trivial torsor over the vector group $\VV^\lor G(Z) \otimes \QQ$ under the vector group $\VV^\lor \Gm(Z)\otimes \QQ$.
\end{ssTheorem}

\begin{ssRemark}
\label{TC13}
When the line bundle $L$ associated to $L^\star$ has degree 0, $L^\star$ is again a semiabelian $Z$-scheme, and Theorem \ref{TC12} is essentially a special case of paragraph \ref{TC11}. For general $L$, as we'll see presently, Theorem \ref{TC12} follows from Theorem 7.2 of joint work with L. Alexander Betts \cite{DCBetts}, concerning the structure of the motivic algebra of a $\Gm$-bundle over any variety; the proof of the latter is somewhat involved.
\end{ssRemark}

\begin{proof}[Proof of Theorem \ref{TC12}]
Theorem 7.2 of \cite{DCBetts} provides a pushout square (lower portion of the following diagram)
\[
\xymatrix{
\QQ(-1)[-2] \ar[r]^{c_1(L)} & 
C^*(G) \ar[r] \ar[d] & 
\ep(L^\star) \ar[r] \ar[d] &
\QQ(-1)[-1] 
\\
& \Sym C^*(G) \ar[d] \ar[r] & 
\Sym \ep(L^\star) \ar[d]
\\
& C^*(G) \ar[r] & 
C^*(L^\star)
}
\]
in which $\ep(L^\star)$ is given by the cofiber (in $
\Dd_M$) of the motivic 1st Chern class of $L$, as shown. Fix $R$. The exact triangle at the top gives rise to a triangle of mapping spectra 
\begin{small}
\begin{align*}
\SHom \big( C^*(G), R(0) \big)
& \from
\SHom \big( \ep(L^\star), R(0) \big)
\from
\SHom \big( R(0), R(1)[1] \big).
\\
\SHom \big( R(0), R(1)[2] \big)
\from &
\end{align*}
\end{small}
This gives $\Hom_{\Cc_M}\big( C^*(L^\star), R \otimes \one \big)$ the structure of a pseudotorsoric
\[
\Hom_{\Dd_M} \big(R(0), R(1)[1] \big)\text{-module}
\]
over $\Hom_{\Cc_M} \big( C^*(G), \one \otimes R \big)$ \cite[\S3]{DCBetts}. Moreover, for nonzero $R$, the homotopy type $\pi_0\Hom_{\Cc_M}\big( C^*(L^\star), R \otimes \one \big)$ is nonempty under our assumption \ref{TC12}(*).
\end{proof}

\subsection{Punctured elliptic curves}
\label{sSec:Ell}

\subsubsection{}
\label{EB12}
Let $Z$ be an open subscheme of $\Spec \ZZ$ and let $X \to Z$ be the complement of the zero section in an elliptic scheme $E \to Z$. We begin by recalling Iwanari's computation of $\fkS_3C^*(X)$ in outline, while observing that the same computations hold integrally. As in paragraph \ref{TC11}, $C^*(E)$ is the free commutative algebra
\[
C^*(E) = \Sym (H^1[-1]) \simeq 
\QQ(0) \oplus H^1[-1] \oplus \QQ(-1)[-2]
\]
on an object
\[
H^1[-1] = H^1(E)[-1] \in \Dd_M. 
\]
The object $H^1$ is compact and satisfies $\Lambda^3 H^1 = 0$. It follows by Proposition 5.2 of Iwanari \cite{Iwanari} that there's a symmetric monoidal colimit preserving functor
\[
\si:\Dd \Rep \GL_{2, \QQ} \to \Dd_M
\]
from the derived \oo-category of $\GL_2$-representations, which sends the standard representation to $H^1$.

We have $C^*(X) \simeq \QQ(0) \oplus H^1[-1]$; the map $C^*(X) \from C^*(E)$ induced by the inclusion corresponds to the projection. We'll abbreviate
\[
C := C^*(X).
\]

\subsubsection{}
\label{EB13}
We argue that the map $\fkS_1C \to C$ is equivalent to the map induced by the projection
\[
\QQ(0) \oplus H^1[-1] \oplus \QQ(-1)[-2]
\surj
\QQ(0) \oplus H^1[-1].
\]
For this, it's enough to show the map
\[
\mu:\QQ(-1)[-2] \to H^1[-1]
\]
is homotopic to 0 in $\Dd_M$. We have
\[
\tag{*}
H_1 \simeq H^1(1),
\]
and we have
\begin{align*}
\tag{**}
\Hom_{\Dd_M} \big(H^1[-1], \QQ(0) \big)
&= \Hom_{\Cc_M} \big(C^*(E), \one  \big)
\\
&= \Aaug(E)
\\
&=E(Z) \otimes \QQ.
\end{align*}
Together,
\begin{align*}
\Hom_{\Dd_M} \big( \QQ(-1)[-2], H^1[-1] \big)
&=
\Hom_{\Dd_M} \big( H_1[1], \QQ(1)[2] \big)
\\
&= \Hom_{\Dd_M} \big( H^1[-1], \QQ(0) \big)
\\
&= E(Z) \otimes \QQ.
\end{align*}
The pullback
\[
\Hom_{\Dd_M(\Spec \QQ)}
\big( H_1[1], \QQ(1)[2] \big)
\xfrom{\iota^*}
\Hom_{\Dd_M(Z)} \big( H_1[1], \QQ(1)[2] \big)
\]
along $\iota: \Spec \QQ \to Z$ corresponds to the isomorphism
\[
E(\QQ) \otimes \QQ \xfrom{\sim}
E(Z) \otimes \QQ.
\]
In particular, it's injective. So it's enough to show $\iota^*\mu \simeq 0$. But this is established by Iwanari. 

\subsubsection{}
\label{EB14}
Hence 
\[
V_1 = \QQ(-1)[-2]
\]
and $\fkS_2C$ is given by the pushout
\[
\tag{*}
\xymatrix{
\Sym \QQ(-1)[-2]
\ar[r] \ar[d]
& \Sym H^1[-1]
\ar[d]
\\
\one
\ar[r]
&
\fkS_2C. 
}
\]
By Proposition 5.2 of Iwanari \cite{Iwanari},
\footnote{The proposition from loc. cit. is stated only for $\Dd_M$ over a field, but in fact the result is proved in far greater generality.}
this pushout may be computed as a homotopy pushout in the model category of $GL_2$-equivariant cdga's, where the underlying graded algebra is given by the tensor product
\begin{align*} 
\fkS_2C 
&= \Sym \big( \QQ(-1)[-1] \big) \otimes \Sym \big( H^1[-1] \big) 
\\ 
&= \big( \QQ(0) \oplus \QQ(-1)[-1] \big) \otimes \big( \QQ(0) \oplus H^1[-1] \oplus \QQ(-1)[-2] \big) 
\\ 
&= \!\begin{aligned}[t]
\QQ(0) \oplus H^1[-1] \oplus H^1(-1)[-2] & \oplus \QQ(-1)[-1] \\ 
   &\quad \oplus \QQ(-1)[-2] \oplus \QQ(-2)[-3] 
   \end{aligned}
\end{align*}
and the underlying complex of $\GL_2$-representations is given by
\[
\QQ(0) \xto{0} 
\QQ(-1)\oplus H^1 \xto{(\m{Id}, 0)} 
\QQ(-1) \oplus H^1(-1) \xto{} \QQ(-2)
\]
placed in cohomological degrees $[0,3]$. Thus, the underlying complex, and hence the underlying object of $\Dd$, is equivalent to 
\[
\fkS_2C \simeq 
\QQ(0) \oplus H^1[-1] \oplus W_2 
\]
where
\[
W_2 = H^1(-1)[-2] \oplus \QQ(-2)[-3].
\]

\sspar{}
\label{EB15}
We argue that the fiber of the map
\[
\tag{*}
\xymatrix
@R=5pt
{
\fkS_2C \ar[r] \ar@{=}[d] 
& C \ar@{=}[d]
\\
\QQ(0) \oplus H^1[-1] \oplus W_2
&
\QQ(0) \oplus H^1[-1]
}
\]
is isomorphic to $W_2$. It's a general fact that given a diagram
\[
\xymatrix{
X \oplus Y \ar[r]^-f & X
\\
X \ar[u]^-{\m{coproj}} \ar[ur]_-g^-\sim
}
\]
in a triangulated category (in which $g$ is iso), the sequence
\[
Y \xto{(-g\inv \circ f, \m{Id}_Y)} X \oplus Y \xto{f} X
\]
is exact. So it's enough to check that the endomorphism of $\QQ(0) \oplus H^1[-1]$ induced by (*) is an isomorphism, which essentially follows from the construction. Thus, we have
\[
V_2 = W_2. 
\]

\sspar{}
\label{EB16}
Gathering what we know about the corresponding K-groups, we find that the lower portion of the tower and dual tower are given by
\begin{footnotesize}
\[
\xymatrix
@ C = 2ex
{
\fkS_3C && 
\Aa^3 \ar[d]
\\
\fkS_2C \ar[u] & \Sym \big(H^1(-1)[-2] \oplus \QQ(-2)[-3]\big) \ar[l]
& \Aa^2 \ar[d] \ar[r]
&\VV^\lor K^{(2)}_{\bullet +1}(X)
\\
\Sym \big(H^1[-1]\big) \ar[u] & \Sym \big(\QQ(-1)[-2]\big)  \ar[l]
& \VV^\lor E(Z)_\QQ \ar[d] \ar[r]
& \Kk(\VV^\lor \Oo^*_{Z, \QQ}, 1) 
\\
\one \ar[u] &
\Sym \big(H^1[-2]\big) \ar[l]
& 
\ast \ar[r]
&
\Kk \big( \VV^\lor E(Z)_\QQ, 1 \big),
}
\]
\end{footnotesize}
where $\VV^\lor K^{(2)}_{\bullet+1}(X)$ denotes a presheaf of pointed homotopy types whose homotopy presheaves are given by
\[
\pi_i \VV^\lor K^{(2)}_{\bullet+1}(X) = 
\VV^\lor K^{(2)}_{i+1}(X).
\]

\sspar{}
\label{EB17}
Applying the proof of proposition \ref{BB37} we find that $\pi_0 \Aa^3$ is representable by a tower of torsors under vector groups as follows:
\[
\xymatrix{
\ast \ar[r] & \VV^\lor K^{(2)}_2(X)
\ar@{}[r]|-{\rotatebox{-90}{\scalebox{2}{$\circlearrowright$}}} & \pi_0 \Aa^3 \ar[d]
\\
\ast \ar[r] & \VV^\lor  \Oo^*_{Z, \QQ}
\ar@{}[r]|-{\rotatebox{-90}{\scalebox{2}{$\circlearrowright$}}}
& \pi_0 \Aa^2 \ar[r] \ar[d] & \ast
\\
&& \VV^\lor E(Z)_\QQ \ar[r] \ar[d] & \ast
\\
&& \ast &.
}
\]

\subsection{Tate-cellular categories}
\label{sSec:TCell}

\sspar{}
\label{KG48}
We return to a general field $F$ of characteristic 0. A mixed Tate Tannakian category over $F$ is a neutral Tannakian category with an object of rank 1 denoted $F(1)$ such that the objects $F(n) := F(1)^{\otimes n}$ represent the simple objects of $T$, and
\[
\Ext^i \big(F(0), F(n) \big) = 0
\quad \text{for all } n<0.
\]
We say $T$ is \emph{locally of finite type} if the extension spaces $\Ext^i \big( F(0), F(n) \big)$ are finite dimensional over $F$, and we say $T$ is \emph{free} if
\[
\Ext^i \big( F(0), F(n) \big) = 0 \quad \text{for } i >1. 
\]
In the presence of a locally free, finite type mixed Tate Tannakian category $T$, we let
\[
E_n = E_n(T) := \Ext^1_T 
\big( F(0), F(n) \big).
\]

\begin{ssExamples}
\label{LA55}
The category of mixed Tate motives over an open integer scheme $Z$ is locally of finite type and free with 
\[
E_n = K_{2n-1}^{(n)}(Z).
\]
Thus, $E_1 = \Oo(Z)^\ast \otimes \QQ$, and for $n>1$, 
\[
\dim_\QQ E_n = 
\begin{cases}
s & n \text{ even} \\
r+s & n \text{ odd},
\end{cases}
\]
where $r$ is the number of real points and $s$ the number of conjugate pairs of complex points.

The category of mixed Tate filtered $\phi$ modules over a $p$-adic field $K_\pP$ is locally of finite type and free with extension spaces 
\[
E_n \simeq K_\pP
\]
of dimension 
\[
\dim_{\Qp} E_n = [K_\pP: \Qp].
\]
\end{ssExamples}

\begin{ssDefinition}
\label{LA57}
We say that the presentably symmetric monoidal $F$-linear \oo-category $\Dd \in \CAlg(\PrLst)_{\Dd(F)/}$ equipped with a t-structure $(\Dd^{\le 0}, \Dd^{\ge 0})$ is \emph{Tate-cellular} if it's equivalent to $\Ind \Dd^b(T)$ for some mixed Tate Tannakian category $T$ over $F$. We import the adjectives from above: $\Dd$ is \emph{locally of finite type} if $T$ is locally of finite type, and \emph{free} if $T$ is free.
\end{ssDefinition}

\begin{ssProposition}
\label{KG50}
Let $T$ be a locally finite type free mixed Tate Tannakian category over $F$, and let $\Dd = \Ind \Dd^b(T)  \in \CAlg (\PrLst)_{\Dd(F)/}$ be the associated Tate-cellular category. As above, we let 
\[
E_n(T) = \Ext^1_T \big( F(0), F(n) \big).
\]
Let $V_{i,j} = F(-i)[-j]$. Then  the associated presheaf of homotopy types 
\[
\Vv_{i,j}: \opnm{Aff}_F\op \to \HoT
\]
\[
\Vv_{i,j}(R) = \Hom_\Dd \big( F(0), R(i)[j] \big)
\]
satisfies
\[
\Vv_{i,j} = 
\begin{cases}
\ast & i < 0 \\
\Kk ( \GG_{a,F}, j) & i = 0 \\
\Kk \big( \EE_i(T), j-1 \big) & i >0,
\end{cases}
\]
where we set
\[
\Kk \big( ?, l \big) := \ast 
\quad \text{if } l<0.
\]
\end{ssProposition}

\begin{proof}
Direct consequence of Lemma \ref{rep4}.
\end{proof}

\subsection{The four line complement}
\label{sSec:4L}

When $X$ is the projective plane minus four lines in general position, the vanishing of $\pi_2$ required in Proposition \ref{BB37} fails already for $n = 2$. We show that our refined criterion, Proposition \ref{BB45} does apply, albeit in a somewhat simplistic way: in the notation of that proposition, $Q_1 = 0$. 

\sspar{}
\label{FA11}
Let 
$
X = \Spec \ZZ[x,y,x\inv, y\inv, (x+y-1)\inv].
$
Then
\[
\tag{*}
C_*(X) = \QQ(0) \oplus \QQ(1)[1]^{\oplus 3} \oplus \QQ(2)[2]^{\oplus 3}.
\]
To see this, we first apply the K\"unneth formula to $\Gm \times \Gm$ to conclude that 
\[
C_*(\Gm \times \Gm) = C_*(\Gm) \otimes C_*(\Gm) = 
\QQ(0) \oplus \QQ(1)[1]^{\oplus 2} \oplus \QQ(2)[2]
\]
and then the Gysin sequence \cite[15.5.4]{Mazza}
\[
C_*(Y \setminus Z) \to C_*(Y) \to C_*(Z)(c)[2c] \to C_*(Y \setminus Z)[1]
\]
to $Y = \Gm \times \Gm$ and $Z = \{x+y = 1\} \simeq \thrpl$
to obtain an exact triangle
\[
C_*(X) \to \QQ(0) \oplus \QQ(1)[1]^2 \oplus \QQ(2)[2] 
\xto{\phi}
\QQ(1)[2] \oplus \QQ(2)[3]^2 \to C_*(X)[1].
\]
Referring to Examples \ref{LA55} ($Z = \Spec \ZZ$) and Proposition \ref{KG50}, we see that $\phi$ occupies a connected mapping space, hence must be homotopic to zero. The formula (*) follows. 

\sspar{}
\label{FA12}
We let $W \in \{M, \Fphi\}$ (motives or filtered $\phi$ modules), and we let $F$ denote the associated coefficient field $\QQ$ or $\Qp$. We prepare for our computation of the second level in the Iwanari tower by observing that maps 
\[
F(i)[j]^{\oplus n} \to F(i)[j]^{\oplus m}
\]
have uninteresting fibers:

\begin{ssProposition}
\label{15j1}
Given $i, j \in \ZZ$, a morphism
\[
f: F(i)[j]^{\oplus n} \to F(i)[j]^{\oplus m}
\]
has a natural notion of \emph{rank}, and if $r$ denotes the rank and we set $k := n - r$ and $q := m-r$, then 
\[
\Fib(f) \simeq F(i)[j]^{\oplus k} \oplus F(i)[j-1]^{\oplus q}.
\]
\end{ssProposition}

\begin{proof}
By Proposition 5.2 of Iwanari \cite{Iwanari} there's a symmetric monoidal colimit preserving functor
\[
\si:\Dd \Rep \GG_{m, F} \to \Dd_W
\]
from the derived \oo-category of $\Gm$-representations, which sends the standard representation to $F(1)$. The map
\[
\Hom_{\Dd \Rep_F(\Gm)}
\big(F(i)[j]^{\oplus n}, \big(F(i)[j]^{\oplus m} \big) 
\to
\Hom_{\Dd_W}
\big(F(i)[j]^{\oplus n}, \big(F(i)[j]^{\oplus m} \big)  
\]
induced by $\si$, is a map of discrete spaces inducing bijections on $\pi_0$. Since the former is equal to the space of $m \times n$ matrices, we may define the rank of $f$ to be the rank of the associated matrix. Moreover, it follows that, up to homotopy, $f$ is in the image of $\si$ and the fiber may be computed in $\Dd \Rep_F(\Gm)$. Since shifts and twists are auto-equivalences, hence preserve limits, we're also reduced to considering the case $i = j = 0$. The remaining computation in the derived category of $F$ vector spaces is straightforward.
\end{proof}

\sspar{}
\label{FA14}
We return to the second level in the Iwanari tower. Let $C = C^*(X)$ denote the motivic cochain algebra. We have
\[
\xymatrix@R = 1ex
{
\fkS_1 C
\ar@{=}[d] \ar[r]^-{f_1}
 & 
C
 \ar@{=}[d]
\\
\Sym^* \big(\QQ(-1)[-1]^3 \oplus \QQ(-2)[-2]^3   \big)
&
\QQ(0) \oplus \QQ(-1)[-1]^3 \oplus \QQ(-2)[-2]^3.
}
\]
Computing in $\Dd \Rep_\QQ(\Gm)$ we find that $\fkS_1 C$ is a sum of shifted Tate motives of the form $\QQ(i)[i]^{\oplus n_i}$ with $i \le 0$. Since $\Hom_{\DM}(\QQ(i)[i],\QQ(j)[j])$ is connected for $i \neq j$, $f_1$ decomposes up to homotopy as a sum of maps 
\[
f_1^i: \QQ(i)[i]^{\oplus n_i} \to \QQ(i)[i]^{\oplus m_i}. 
\]
Let $r_i$ denote the rank of $f_1^i$ and set $q_i:= m_i - r_i$ and $k_i = n_i - r_i$. The first summand $f_1^{-1}$ is an isomorphism, so that $q_{-1} = k_{-1} = 0$. Applying Proposition \ref{15j1}, we find that 
\[
V_1 \simeq \bigoplus_{i \le -2}
\QQ(i)[i]^{\oplus k_i} \oplus \QQ(i)[i-1]^{\oplus q_i}.
\]
Moreover, we have $n_{-2} = 6$ and $m_{-2} = 3$ so $k_{-2} \ge 3$ (actually $=3$).

\sspar{}
\label{FA15}
We now treat the motivic or filtered $\phi$ settings by considering realization of the motivic Iwanari tower into any Tate-cellular category.

\sspar{}
\label{KH45}
We've seen that the lower portion of the motivic Iwanari tower and the dual tower of presheaves of augmentations in any Tate-cellular realization take the following form:
\[
\begin{small}
\tag{*}
\xymatrix
@ C = 1ex
{
\QQ(0) \oplus \QQ(-1)[-1]^{\oplus 3} \oplus \QQ(-2)[-2]^{\oplus 3} 
& \vdots \ar[d] 
\\
\vdots \ar[u] & \Aa^2 \ar[d]  
\\
\fkS_2 C \ar[u] &  
\Kk (\EE_1, 0)^3 \times \Kk (\EE_2, 1)^3 \ar[r] 
& \Kk (\EE_2, 1)^3 \times \Vv' 
\\
\Sym \big( \QQ(-1)[-1]^{\oplus 3} \oplus \QQ(-2)[-2]^{\oplus 3} \big) \ar[u] 
& \Sym \big( \QQ(-2)[-2]^{\oplus 3} \oplus V' \big) \ar[l]  
}
\end{small}
\]
with $V'$ a sum of shifts of Tate motives $\QQ(-i)[-i]$ with $i \ge 2$, so that $\Vv'$ is connected and simply connected. 

Our goal for this section is to prove the following proposition. 

\begin{ssProposition}
\label{KL57}
We have
\[
\pi_0 \Aa^2 \xto{\sim} \pi_0 \Aa^1 \simeq \EE_1^3.
\]
\end{ssProposition}

The proof spans paragraphs \ref{KI44} -- \ref{KJ48}.

\sspar{}
\label{KI44}
We repeat the construction of $Q=Q_1$ in this concrete context. By paragraph \ref{KH45}, we have a group action in presheaves of sets
\[
\xymatrix{
\EE_2^3 
\ar@{}[r]|-{\rotatebox{-90}{\scalebox{2}{$\circlearrowright$}}}
&
\pi_0\Aa^2 \ar@{->>}[d]^-{\pi_0\phi}
\\
& \EE_1^3.
}
\]
Given $x: \ast \to \Aa^2(R)$ for some $F$-algebra $R$, the stabilizer of the connected component $[x]$ of $x$ in $\Aa^2(R)$ may be identified with the image of the map of abelian groups 
\[
\tag{*}
\begin{small}
\xymatrix
@R = 1ex @C = 2ex
{
\pi_1 \big( \Kk (R \otimes_F E_1, 0)^3 \times 
\Kk (R \otimes_F E_2, 1)^3, \phi x \big)
\ar[r] \ar@{=}[d]
&
\pi_1 \big( \Kk ( R\otimes_F E_2, 1)^3 \times \Vv'(R), 0 \big)
\ar@{=}[d]
\\
H_1 \big( [\phi x], \ZZ \big)
\ar[r]
&
(R \otimes_F E_2)^3.
}
\end{small}
\]
In particular, the map (*) depends only on the choice of a connected component
\[
\xi = [\phi x] \in \pi_0
\big( \Kk (R \otimes_F E_1, 0)^3 
\times \Kk (R \otimes_F E_2, 1)^3\big)
= R \otimes_F E^3_1. 
\]
Varying $\xi$, we obtain a commutative diagram
\[
\xymatrix{
R \otimes_F E_1^3 \times R \otimes_F E_2^3
\ar[rr] \ar[dr]
&&
R \otimes_F E_1^3 \times R \otimes_F E_2^3
\ar[dl]
\\
& 
R \otimes_F E_1^3 
}
\]
which is natural in $R$ and additive in fibers, hence a map of vector groups over $\EE_1^3$ ($\psi$ below)
\[
\xymatrix{
\EE^3_1 \times \EE_2^3
\ar[r]^-\psi
&
\EE^3_1 \times \EE_2^3 
\ar[r]
&
Q
\ar@{}[r]|-{\rotatebox{-90}{\scalebox{2}{$\circlearrowright$}}}
&
\pi_0\Aa^2 \ar@{->>}[d]^-{\pi_0\phi}
\\
&&& \EE_1^3
}
\]
whose presheaf cokernel $Q$ acts freely and transitively on the fibers of $\pi_0 \phi$.

\sspar{}
\label{KJ44}
We claim $Q = 0$. For this, it's enough to fix an arbitrary $F$-point $e \in \EE_1^3(F) = E_1^3$ and to show that the fiber of $\psi_e$ over $e$ has full rank.

\begin{ssRemark}
A priori, one may expect the rank of $\psi_e$ to depend on the higher coherences encoded in $C^*(X)$. We will argue, however, that the rank of 
\[
\psi_e : E_2^3 \to E_2^3
\]
depends only on the multiplication map
\[
C^*(X) \otimes C^*(X) \to C^*(X),
\]
a morphism in $\Dd \Rep \Gm$. 
\end{ssRemark}

\sspar{}
\label{KJ45}
It will be convenient to abbreviate
\[
H := \QQ(-1)[-1]^{\oplus 3}, \quad 
I := \QQ(-2)[-2]^{\oplus 3},
\]
while noting that there's an equivalence
\[
\tag{$\sigma$}
I \overset \si \simeq \Sym^2 H
\]
of 3-dimensional vector spaces placed in weight $4$ and cohomological degree $2$ in $\Dd \Rep \Gm$. With this notation, the motivic cochain algebra (absent its algebra structure) is given by
\[
C = C^*(X) = \QQ(0) \oplus H \oplus I;
\]
the linear Iwanari truncation is given by
\[
\fkS_1 = \Sym \big( H \oplus I \big).
\]
We denote the fiber of $\fkS_1 \to C$ by $V$. As noted above, the latter may be computed in $\Dd \Rep \Gm$. We've fixed an augmentation 
\[
e: \Sym H \to \one
\]
from which we obtain a map of commutative algebras
\[
\Sym(H \oplus I) \to \Sym I.
\]
The object $V \in \Dd$ contains $I$ as a canonical direct summand. Composing, we obtain a map
\[
\tag{*}
\Sym I \from \Sym (H \oplus I) \from \Sym V \from \Sym I. 
\]
Recalling that $\Hom_\Dd \big( I , F(0) \big)$ is connected with abelian fundamental group, our task is to show that the induced map of $F$-\vectorspaces
\[
\pi_1 \Hom_\Dd \big( I , F(0) \big)
\to
\pi_1 \Hom_\Dd \big( I , F(0) \big)
\] 
has full rank. 

\sspar{}
\label{KJ47}
We claim that 
the map of symmetric algebras \ref{KJ45}* sends $I \subset \Sym I$ into $I \subset \Sym I$ up to homotopy, and that the resulting map $I \to I$ is an isomorphism. 

\sspar{}
\label{KJ48}
The symmetric algebra $\Sym (H \oplus I)$ evidently contains canonical copies of $I$ and $\Sym^2 H$ as summands. Moreover, we've seen that the map
\[
\fkS_1 = \Sym(H \oplus I) \to C = \QQ(0) \oplus H \oplus I
\]
decomposes into a direct sum of maps: one direct summand is the evident map $I \to I$ (the identity); another is a map 
\[
 \Sym^2 H \xto{\si} I
\]
induced by the multiplication map of $C$. This map may be seen to be an equivalence, for instance, by considering the Betti realization, and the known identification of the cohomology algebra of $X(\CC)$ as the free graded-commutative algebra on generators $a,b,c$ in cohomological degree $1$ modulo the one relation $abc = 0$. 

This allows us to pin down the factor $I$ occurring in $V$ as the image of the map
\[
I \xto{-\si\inv \oplus id} (\Sym^2 H) \oplus I.
\]

We claim that up to homotopy, the map
\[
\Sym^2 H \subset \Sym( H \oplus I) \to \Sym I
\]
is equivalent to 0. Indeed, this map is induced by the tensor product of $e$ with itself :
\[
e \otimes e : \QQ(-2)[-2]^{\oplus 9} \to \QQ(0).
\] 
The latter belongs to a cartesian power of the spaces 
\[
\Hom_\Dd  \big( F(0), F(2)[2] \big)
\]
which are connected. Thus
\[
e \otimes e \sim 0.
\]

It follows that the composition
\[
I \to \Sym^2 H \oplus I  \subset \Sym(H \oplus I) \to \Sym I
\]
lands in $I$ and induces the identity map on $I$. This establishes the claim made in paragraph \ref{KJ47}, hence the claim made in paragraph \ref{KJ44}, and hence Proposition \ref{KL57}. \qed

\section{Compatibility with Chabauty-Kim theory}
\label{Sec:Comp}

\subsection{Review of cdga's, Hirsch extensions and minimal models}
\label{Sec:Anim}
We review generalities on cdga's. These will be used in section \ref{Sec:CompFil} to compare our ``coherent'' filtrations with the concrete models considered classically.

\subsubsection{Setup}
\label{LA22}
We fix $F$ a field of characteristic $0$. Let $\Cpx(F)$ denote the category of (unbounded) complexes of vector spaces (of arbitrary dimension). Let $\cdga(F) = \CAlg \Cpx(F)$, the category of \emph{cdga's}. In both categories $\Cpx(F)$ and $\cdga(F)$, we define a morphism to be a \emph{weak equivalence} if it's a quasi-isomorphism and a \emph{fibration} if it's surjective; see Olsson \cite{OlssonBar} and the references there. With these definitions, $\Cpx(F)$ is a symmetric monoidal model category, and the symmetric algebra functor and the forgetful functor form a Quillen adjunction with $\Sym$ left adjoint. The underlying (symmetric monoidal) \oo-categories are equivalent to the derived \oo-category $\Dd(F)$ and to its category of commutative algebras $\Cc(F) = \CAlg \Dd(F)$. Our assumption that $F$ be of characteristic zero guarantees $\Sym$ takes quasi-isomorphisms to quasi-isomorphisms, hence the square of \oo-categories
\[
\xymatrix{
\Cpx(F) \ar[r]^-{\Sym} \ar[d] & \cdga(F) \ar[d]
\\
\Dd(F) \ar[r]_{\Sym} & \Cc(F)
}
\]
commutes up to homotopy. 

\subsubsection{Free cdga's}
\label{LA23}
Let $E \in \Cpx(F)$ be a complex. The symmetric algebra $\Sym E \in \cdga(F)$ is determined concretely by the following properties: the underlying graded-commutative algebra is the free graded-commutative algebra
\[
\Sym E = \Sym(E^\m{even}) \otimes 
\La(E^\m{odd});
\]
the differential is determined from the given differential on $E$ by the Leibniz rule. The free cdga $\Sym E$ comes with a canonical augmentation $\Sym E \to F$. 

\subsubsection{The cdg subalgebra generated by a complex}
\label{LX66}
Suppose given an inclusion $A \subset B$ of cdga's, and a subcomplex $E \subset B$. Then the graded subalgebra $A\langle E \rangle$ generated by the graded vector subspace $A + E + dE \subset B$ is a cdg subalgebra. 

\subsubsection{Hirsch extensions}
\label{LA24}
Now suppose given a complex $E \in \Cpx(F)$, a cdga $C \in \cdga(F)$ and a morphism of complexes 
\[
\al: E \to C. 
\]
The \emph{Hirsch extension $C \langle \int E \rangle$ of $C$ by $E$} is a concrete model for the homotopy pushout
\[
\xymatrix{
\Sym E \ar[r] \ar[d] & C \ar[d] 
\\
F \ar[r] & C \langle \int E \rangle
}
\]
of the induced map $\Sym E \to C$ in $\cdga(F)$ by the canonical augmentation $\Sym E \to F$, which may be constructed as follows. After possibly replacing $E$ by a quasi-isomorphic complex, we may assume $d=0$. Let $D(E)$ be the graded vector space $E \oplus E[1]$ with differential 
\[
E^{i-1} \oplus E^{i} \to E^{i} \oplus E^{i+1}
\]
given by the second projection. Then the evident morphisms
\[
E \to D(E) \to 0
\]
factor the zero map into a cofibration followed by a trivial fibration, and so 
\[
\Sym E \to \Sym D(E) \to F
\]
factors the canonical augmentation as a cofibration followed by a trivial fibration. The homotopy pushout is thus computed by the relative tensor product 
\begin{align*}
C \langle \int E \rangle 
& \simeq C \otimes_{\Sym E} \Sym D(E)
\\
& \simeq
\coeq \Big(C \otimes_F \Sym E \otimes_F \Sym D(E)
\rightrightarrows
C \otimes_F \Sym D(E) \Big)
.
\end{align*}
In turn, this is given by the graded algebra
\[
C \langle \int E \rangle 
\simeq
C \otimes_F \Sym \big( E[1] \big)
\]
equipped with the unique differential which extends the given one on $C$ and sends $x \in E$ to $dx := \al(x)$. 

\subsubsection{Minimal models}
\label{LA25}
A cdga $A$ is \emph{connected} if $A^0 = F$ is the ground field and $A^i = 0$ for $i<0$. When this is the case, the underlying complex decomposes under the unit map as
\[
A \simeq F \oplus \tilde A,
\]
and we define the \emph{indecomposables of $A$} by
\[
I(A) = \frac {\tilde A}  {\tilde A \cdot \tilde A}.
\]
The cdga $A$ is \emph{minimal} if it's connected, and if it may be written as a nested union $A = \bigcup A_i$,
\[
F = A_0 \subset A_1 \subset A_2 \subset \cdots
\]
with $A_i \subset A_{i+1}$ a Hirsch extension by a complex concentrated in one degree, and with induced differential
\[
d: I(A) \to I(A)
\]
equal to zero (see e.g. \cite[Definition 5.2]{Morgan}). 

A cdga $A$ is \emph{cohomologically connected} if $H^iA = 0$ for $i<0$ and $H^0A = F$. A morphism of cdga's is a \emph{weak equivalence} (or simply an ``equivalence'') if the morphism of underlying complexes is a quasi-isomorphism. By Sullivan \cite{Sullivan} (or see Bousfield--Guggenheim \cite{BousfieldGug}), every cohomologically connected cdga $A$ admits a \emph{minimal model}, i.e. an equivalence of cdga's $M \to A$ with $M$ minimal.

\subsubsection{Notation for complexes}
\label{NC55}
If $E$ is a cochain complex of $F$-\vectorspaces, we occasionally regard $E$ as a graded vector space with differential $d: E \to E$ of graded degree $1$. We then denote $ZE = \ker d$ and $BE = \Im d$, both regarded as graded vector subspaces of $E$. With this notation, $Z^iE$ denotes the $i$th graded piece of $ZE$, and similarly for $B^iE$. We also define
\[
\gr^iE := E^i, 
\]
the $i$th graded piece of $E$. (This redundant notation helps to avoid a buildup of indexes when working with filtered complexes.)

\subsection{From augmentations to $\pi_1$-torsors (in families)}
\label{sSec:tors}

\sspar{}\label{DA22}
We work over an open integer scheme $Z \subset 
\Spec \Oo_K$. We fix a closed point $\pP \in Z$. Let $Z_\pP = \Spec \Oo_\pP$ be the complete local scheme at $\pP$ and let $K_\pP$ be the complete local field at $\pP$, a finite extension of $\Qp$. We let $\MFphi$ denote the $\Qp$-Tannakian category of admissible filtered $\phi$ modules over $K_\pP$. In paragraphs \ref{DA33} -- \ref{DA39} we let
\[
\Dd_M := \Dd_M(Z_\pP) = \DM^\m{il}(Z_\pP, \QQ)
\]
denote the symmetric monoidal \oo-category of Ind-lisse motives, and we fix a morphism 
\[
x: C \to \one
\]
in $\Cc_M(Z_\pP):= \CAlg \Dd_M(Z_\pP)$ (equivalently, an object of $\Cc_M^\m{aug} := \CAlg^\m{aug}\Dd_M \simeq \Alg^\m{aug}\CAlg \Dd_M$).

\sspar{}
\label{DA33}
There's an exact monoidal realization functor on the level of homotopy categories
\footnote{It appears that an \oo-categorically enhanced filtered $\phi$ realization functor \textit{can} be extracted from the literature. This helps to fill in the landscape with a few more coherently commuting squares. However, this isn't strictly needed for Theorem \ref{IN13}, since in establishing the commutativity of the diagram there, we may pass to homotopy categories immediately after applying Koszul duality to obtain cotorsors under the bar construction.}
\[
\Re_{F\phi}: h\Dd_M \to h\Dd_{F\phi}
\]
where $\Dd_{F\phi} = \Dd(\Ind \MFphi)$. We denote filtered $\phi$ realization also by adding `$F\phi$' as a decoration (superscript or subscript). We let 
\[
{_xB_x} = {_xB_x}(C) = \mathfrak{D}_\m{Ksz}(C, x),
\]
the Koszul dual coalgebra \cite[\S 5.2.2]{LurieHA}. Let $\om: \Ind \MFphi \to \Vect \Qp$ be a $\Qp$-rational fiber functor. We assume $\Re_{\Fphi}(C)$ is cohomologically connected and we let
\[
\tag{*}
\pi(C,x)^{F\phi} = \Spec \om H^0 \Re_{F\phi} ({_xB_x} ),
\]
a proalgebraic (in fact prounipotent) $\Qp$-group with an action of the Tannakian Galois group $G_\om(\MFphi)$.

\sspar{}\label{DA34}
We define $\Aug(C)$ to be the presheaf of sets
\[
\Aug(C)(R) = \pi_0 \Hom_{\Cc}(C, \one \otimes R),
\]
pointed by our fixed augmentation $x$. If $R$ is a $\Qp$-algebra, then, by definition, 
\[
\Aug(C)_\Qp(R) = \Aug(C)(R).
\]
We define the presheaf of pointed sets
\[
\mathbf{H}^1\big(G_\om(\MFphi), \pi(C,x)^{F\phi}\big)
:
\m{Aff}^\m{op}_\Qp \to \Set_*
\]
by
\[
\mathbf{H}^1\big(G_\om(\MFphi), \pi(C,x)^{F\phi}\big)(R) = 
H^1\big(G_\om(\MFphi)_R, \pi(C,x)^{F\phi}_R\big).
\]
Our goal in this section is to construct a map of presheaves
\[
\tag{$\ast_\om$}
\Aug(C)_{\Qp} \to \mathbf{H}^1\big(G_\om(\MFphi), \pi(C,x)^{F\phi}\big).
\]
We begin in paragraphs \ref{DA35}--\ref{DA36} with an outline of the construction.

\sspar{}\label{DA35}
We fix $R$ and we claim that all of the constructions that follow may be made suitably natural in $R$. Fix $\xi : C \to \one \otimes R$. Let
\[
\Dd^M_R:= \Mod_{\one \otimes R}\Dd_M,
\quad 
\Cc^M_R := \CAlg \Dd^M_R \simeq \Cc_{\one \otimes R/}, 
\quad
\Cc_R^\m{M,aug} := \CAlg^\m{aug} \Dd^M_R.
\]
Similarly, let
\[
\Dd^{F\phi}_R:= \Mod_{\one \otimes R}\Dd^{F\phi},
\quad 
\Cc^{F\phi}_R := \CAlg \Dd^{F\phi}_R \simeq \Cc^{F\phi}_{\one \otimes R/}, 
\quad
\Cc_R^{F\phi,\m{aug}} := \CAlg^\m{aug} \Dd^{F\phi}_R,
\]
\[
\Ind \opnm{MF}^\phi_R := \Mod_{\one \otimes R}^\heartsuit (\Ind\MFphi).
\]
(The $\heartsuit$ is there only to emphasize that we consider modules in the symmetric monoidal \textit{abelian} category $\Ind \MFphi$, and \textit{not} its derived category.) There's an equivalence of symmetric monoidal $R$-linear categories
\[
\Ind\opnm{MF}^\phi_R 
\simeq 
\Ind\Rep_R \big(G_\om(\MFphi)_R\big)
\]
which commutes with the evident fiber functors to $\Mod^\heartsuit(R)$.

\sspar{}\label{DX33}
We have an equivalence in $\coAlg(\Cc_R)$:
\[
\tag{*}
{_{x \otimes R} B_{x \otimes R}}(C \otimes R) \simeq 
{_x B_x} \otimes R.
\]
Moreover, the operations $\Re_{F\phi}$, $H^0$ are both compatible with base-change, so that we obtain an equivalence 
\[
H^0 \Re_{F\phi} ({_x B_x} \otimes R)
\simeq
H^0 \Re_{F\phi} ({_x B_x}) \otimes R
\]
in $\coAlg \CAlg \Ind \opnm{MF}^\phi_R$

\sspar{}\label{DA36}
The $(\Spec R)$-family of augmentations $\xi$ gives rise to an augmentation
\[
\xi:C \otimes R \to \one \otimes R
\]
in $\Cc_R^M$. Further, the data $(C\otimes R, x \otimes R, \xi)$ gives rise to an object of $\LMod^\m{aug}(\Cc_R^M)$ (We view $\one \otimes R$ as $C\otimes R$-module via $\xi$). Applying Koszul duality for left modules \cite{LMKoszul, brantner2025pd}, we obtain a torsoric ${_x B_x} \otimes R$-comodule
\[
\tag{*}
{_{x\otimes R} B_\xi} = {_{x\otimes R} B_\xi}(C \otimes R) := \Dd_\m{Ksz}
(C\otimes R, x \otimes R, \xi)
\]
in $\Cc_R^M$. Let ${_xB_x^{F\phi} } = \Re_{F\phi} ({_xB_x})$. Applying \textit{filtered $\phi$}-realization to (*), we obtain a ${_x B^{F\phi}_x}\otimes R$-cotorsor ${_{x\otimes R}B_\xi^{F\phi}}$ in the homotopy category $h (\Cc^{F\phi}_R)$. Applying truncation into the heart, we obtain an $H^0({_x B^{F\phi}_x}\otimes R)$-cotorsor $H^0({_{x\otimes R}B_\xi^{F\phi}})$ in $\CAlg \Ind  \opnm{MF}^\phi_R$
\footnote{Following Deligne's ``algebraic geometry in a Tannakian category'', we refer to such an object as a \emph{$\Spec H^0({_x B^{F\phi}_x}\otimes R)$-torsor in $\opnm{MF}^\phi_R$}. See Definition \ref{DA40}.}
, hence an element of $H^1\big(G_\om(\MFphi)_R, \pi(C,x)^{F\phi}_R\big)$. See below for the base-change $\opnm{MF}^\phi_R$.

\medskip

We now revisit aspects of the construction that require explanation. 

\begin{ssLemma}
\label{DA37}
If $R$ is a ring and $G$ is a group-scheme over $R$, we denote by $\Rep_R(G)$ the category of representations of $G$ in (arbitrary) $R$-modules. Let $F$ be a ring, let $G$ be an affine group scheme over $F$, and let $R$ be an $F$-algebra. Then there's an equivalence of categories 
\[
\Mod_{\one \otimes R}(\Rep_F G) \simeq \Rep_R(G_R)
\] 
commuting with the evident forgetful functors to the category $\Mod_R$ of $R$-modules. 
\end{ssLemma}

\begin{proof}
This is all just a sequence of diagram chases; we nevertheless include the details. The object $\one \otimes R \in \Rep_k G$ is given by $R$ with trivial $G$-action, and we'll denote it simply by $R$. We'll construct functors
\[
\Mod_{R}(\Rep_F G) 
\underset{\Psi}{\overset{\Phi}{\leftrightarrows}} 
\Rep_R(G_R).
\]

Let $A:= \Oo(G)$. An $R$-module in $\Rep_F G$ consists of an $F$-module $V$ plus maps of $F$-modules $A \otimes_F V \xfrom{\rho_V} V$, and $R \otimes_F V \xto{\al} V$ such that the evident squares
\[
\xymatrix{
&
				& A \otimes_k R \otimes_k V
				\ar[r]^-{\m{id}_A \otimes \al}
				& A \otimes_F V
\\
A \otimes_F V \ar[d]
& V \ar[l]_-{\rho_V} \ar[d] 
				& R\otimes_F V 
				\ar[r]^-\al 
				\ar[u]^-{\rho_{R\otimes_k V}} 
				& V \ar[u]_-{\rho_V}
\\
A\otimes_F A\otimes_F V 
& A\otimes_F V \ar[l]
				& R \otimes_F R \otimes_F V \ar[r] \ar[u] 
				& R \otimes_F V \ar[u]
}
\]
expressing associativity of the coaction of $A$ on $V$ (left), associativity of the action of $R$ on $V$ (below, right), and equivariance of the action of $R$ on $V$ (above, right), commute. Since the coaction of $A$ on $R$ is given by $\rho_R = 1_A \otimes \m{id}_R$, the coaction on $R\otimes_F V$ is given by
\[
\rho_{R\otimes_F V} = \m{id}_R \otimes \rho_V.
\]

We define the underlying $F$-module of $\Psi(V)$ to be $V$. The map $\al$ endows $V$ with the structure of an $R$-module and enables us to regard $\rho_V$ as a map
\[
\tag{*}
(A \otimes_F R) \otimes_R V 
\simeq A \otimes_F V \xfrom{\rho_V} V.
\]
We claim (*) defines a representation of $G_R$ over $R$. Suppose given $r \in R$, $v \in V$. Then
\begin{align*}
\rho_V(rv) &= \rho_V \big( \al(r \otimes v) \big)
\\
&= (\m{id}_A \otimes \al)\big((\m{id}_R \otimes \rho_V)(r\otimes v) \big)
\\
&= (\m{id}_A \otimes \al) \big( r \otimes \rho_V(v) \big)
\\
&= r \rho_V(v),
\end{align*}
which shows that (*) is $R$-linear. The associativity of $\rho_V$ over $R$ follows from the associativity of $\rho_V$ over $F$. 

Thus, $V$ together with the coaction (*) defines an object $\Psi(V) \in \Rep_R (G_R)$. As for morphisms, if $\phi: V \to W$ is a morphism of $R$-modules in $G$-representations, then $\phi$ itself (i.e. regarded as a map of $F$-modules) is $R$-linear and compatible with the coactions by $A \otimes_F R$, hence a morphism of $G_R$-representations $\Psi(V) \to \Psi(W)$ over $R$.

We turn to the construction of $\Phi$. Suppose given a representation $(A \otimes_F R) \otimes_R V \xfrom{\rho_V} V$ of $G_R$ on an $R$-module $V$. Composing again with the isomorphism $A \otimes_F V \simeq (A \otimes_F R) \otimes_R V$, we obtain a $F$-linear map
\[
A \otimes_F V \xfrom{\rho_V^\Phi} V,
\]
while the structure of $R$-module includes the data of a $F$-linear map $\al: R\otimes_F V \to V$. We claim the data $(V, \rho_V^\Phi, \al)$ defines an $R$-module in $G$-representations over $F$. 

The associativity of the $R$-action defined by $\al$ is given. The $R$-linearity of the coaction $\rho_V$ is equivalent to the $G$-equivariance of $\al$. Denoting the comultiplication of $A$ by $\mu$ and the induced comultiplication of $A \otimes_F R$ by $\mu_R$, the commutativity of the squares
\[
\xymatrix{
(A \otimes_F R) \otimes_R V
\ar@{=}[d] \ar[r]^-{\mu_R \otimes \m{id}_V}
&
(A\otimes_F R)\otimes_R(A\otimes_FR)\otimes_R V
\ar@{=}[d]
&
(A \otimes_F R) \otimes_R V
\ar@{=}[d] \ar[l]_-{\m{id}_{A\otimes_F R} \otimes \rho_V}
\\
A \otimes_F V
\ar[r]_-{\mu \otimes \m{id}_V}
&
A \otimes_F A\otimes_F V
&
A \otimes_F V
\ar[l]^-{\m{id}_A \otimes \rho_V^\Phi}
}
\]
shows that the associativity of the coaction of $A\otimes_F R$ on $V$ via $\rho_V$ is equivalent to that of $A$ on $V$ via $\rho_V^\Phi$.

Evidently, $\Phi$ and $\Psi$ are quasi-inverse.
\end{proof}

\sspar{}\label{IA55}
We turn to the compatibility of Koszul duality with base change. We first clarify our formalism of base-change (\ref{DX33}*). Let $F$ be a ring, let $\Dd(F) \in \CAlg (\m{Pr}^L_\m{st})$ be the unbounded derived \oo-category of complexes of $F$-modules, and let
\[
\Dd \in \CAlg \Mod_{\Dd(F)} \m{Pr}^L_\m{st}
\simeq
\CAlg (\m{Pr}^L_\m{st})_{\Dd(F)/}
\]
be a presentably symmetric monoidal $F$-linear stable $\infty$-category. Then evidently $\Dd$ comes equipped with a symmetric monoidal functor 
\[
\Dd(F) \to \Dd,
\]
hence a functor 
\[
\CAlg \Dd(F) \to \Cc := \CAlg \Dd
\]
\[
R \mapsto R \otimes \one.
\]
Since the real numbers are here nowhere in sight, we allow ourselves to abbreviate $\RR:= R \otimes \one$. As in paragraph \ref{DA35}, we set
\[
\Dd_R := \Mod_R(\Dd)
\qandq
\Cc_R := \CAlg \Dd_R \simeq \Cc_{\RR/}.
\]
Tensor product with $R$ defines a symmetric monoidal functor 
\[
\Dd \to \Dd_R
\]
hence a functor
\[
\Cc \to \Cc_R
\]
\[
\AA \mapsto R \otimes \AA = \RR \otimes \AA.
\]

\sspar{}\label{IA56}
For any monoidal category $\Ee$ we denote $\Ee_\bullet = \Ee_{/\one}$.\footnote{With this notation it is hopefully clear that $\Cc_\bullet\op = (\Cc_{/\one})\op$ and \emph{not} $(\Cc\op)_{/\one}$.} Then the functoriality of twisted arrow categories (together with the equivalence $\Alg(\Cc_\bullet) \simeq \Cc_\bullet$) gives rise to a coherently commuting square of \oo-categories
\[
\xymatrix{
f \ar@{|->}[r] & f_R
\\
\Alg \Tw (\Cc_\bullet) \ar[d] \ar[r] &
\Alg \Tw (\Cc_{R,\bullet}) \ar[d] 
\\
\Cc_\bullet \times \Alg(\Cc_\bullet\op) \ar[r] &
\Cc_{R,\bullet} \times \Alg(\Cc_{R, \bullet}\op)
\\
(A,B) \ar@{|->}[r] & (R \otimes A, R\otimes B)
}
\]
Let $C \in \Cc_\bullet$ be an augmented commutative algebra and let $f \in \Alg \Tw(\Cc_\bullet)$ be a left universal object lying over $\big(C, \DKsz(C)\big) \in \Cc_\bullet \times \Alg(\Cc_\bullet\op)$. Then $f_R$ corresponds to a twisted homomorphism 
\[
R \otimes C \to R \otimes \DKsz(C).
\]
By the universal mapping property of $\DKsz(R \otimes C)$, $f_R$ factors essentially uniquely through a map of coalgebras 
\[
\tilde f_R: \DKsz(R\otimes C) \to R \otimes \DKsz(C)
\]
in $\Cc_{R,\bullet}$.

\sspar{}\label{IA57}
After forgetting the coalgebra structures as well as the commutative algebra structures, $\tilde f_R$ is a map of colimits (we write $C_R := R \otimes C$ and we denote $n$-fold tensor power over $R$ by $\otimes_R^n$)
\[
\colim_{n \in \Delta\op} (C_R)^{\otimes_R ^n}
\to
R \otimes \colim_{n \in \Delta\op} C^{\otimes n}
\simeq
\colim_{n \in \Delta\op} R \otimes  C^{\otimes n}
\]
which comes from a map of diagrams
\[
\Big\{
(C_R)^{\otimes_R ^n} 
\xto{f_R^n}
R \otimes C^{\otimes n}
\Big\}_{n \in \Delta\op}
\]
in which $f_R^n$ witnesses the essentially canonical equivalence
\[
(C_R)^{\otimes_R ^n} 
\simeq
R \otimes C^{\otimes n}.
\]
This completes our justification of \ref{DX33}(*). 

\medskip

Lemma \ref{DA37} allows us to drop the choice of fiber functor from our formulation of \ref{DX33}(*), as we now explain.

\begin{ssDefinition}
\label{DA40}
Let $T$ be a Tannakian category over a field $F$. If $R$ is an $F$-algebra, we set
\[
T_R = \Mod_{R \otimes \one}(\Ind T).
\]
Now let $\pi = \Spec A$ be a prounipotent group object in $T$. Then we have the prounipotent group object $\pi_R := \Spec R \otimes A$ of $T_R$ and we define $H^1(T_R, \pi_R)$ to be the set of isomorphism classes of $\pi_R$-torsors in $T_R$. Finally, we define the presheaf of pointed sets
\[
\mathbf{H}^1(T, \pi): \Aff_F\op \to \Set_\ast
\]
by $\mathbf{H}^1(T, \pi)(R) = H^1(T_R, \pi_R)$.
\end{ssDefinition}

\begin{ssRemark}
\label{DA38}
In view of Lemma \ref{DA37}, the map \ref{DA34}($\ast_\om$) corresponds to a map
\[
\tag{*}
\Aug(C)_\Qp \to 
\mathbf{H}^1\big(\MFphi, \pi(C,x)^{F\phi} \big).
\]
\end{ssRemark}

\subsubsection{Variant}
\label{DA39}
Fix an algebraic closure $\bar K_\pP$ of $K_\pP$ with associated total Galois group $G_{K_\pP}$, let $\Rep^f G_{K_\pP}$ denote the $\Qp$-Tannakian category of crystalline representations of $G_{K_\pP}$ over $\Qp$, and let
\[
\Dd_\et(Z_\pP) := \Ind \Dd^b \Rep^f G_{K_\pP}.
\]
We have the presheaf of pointed sets
\[
\mathbf{H}^1_f \big( 
G_{K_\pP}, \pi(C,x)^\et
\big):=
\mathbf{H}^1_f \big( 
\Rep^f G_{K_\pP}, \pi(C,x)^\et
\big):
\Aff_{\Qp}\op \to \Set_*.
\]
We replace realization into filtered $\phi$ modules by realization
\[
h\Dd_M \to h \Dd_\et(Z_\pP)
\]
into crystalline Galois representations. The triangulated $p$-adic Hodge theory of Deglise-Niziol \cite{DegliseNiziol} provides a commuting diagram of presheaves of pointed sets
\[
\xymatrix{
\Aug(C)_\Qp \ar[d] \ar[dr] \\
\mathbf{H}^1_f \big( 
G_{K_\pP}, \pi(C,x)^\et
\big)
\ar[r]
&
\mathbf{H}^1\big(\MFphi,  \pi(C,x)^{F\phi} \big).
}
\]

\subsubsection{Variant}
\label{DA41}
Instead of working with the category $\Dd_M(Z_\pP)$ of Ind-lisse motives over the trait $Z_\pP$, we work with the category $\Dd_M(Z)$ of Ind-lisse motives over our open integer scheme $Z$. We also fix an algebraic closure $\bar K$ of $K$ as well as an embedding of $\bar K$ in $\bar K_\pP$. We let $Z^o \subset Z$ be the complement of all primes above $p$ (our fixed prime $\pP$ among them) and we let $G_{Z^o} = \pi_1^\et(Z^o)$ with base-point given by $\overline K$. We let $\Rep^f(G_{Z^o})$ be the $\Qp$-Tannakian category of representations of $G_{Z^o}$ over $\Qp$ which are crystalline at primes above $p$, and we let $\Dd_{\et}(Z) = \Ind \Dd^b \Rep^f G_{Z^o}$. Let $C \to \one$ be an augmented algebra in $\Dd_M(Z)$ and let $C_{Z_\pP} \to \one$ be its pullback to $\Dd_M(Z_\pP)$. We have a commuting square of presheaves of pointed sets
\[
\tag{*}
\xymatrix{
\Aug(C)_\Qp \ar[r] \ar[d] & \Aug(C_{Z_\pP})_\Qp \ar[d]
\\
\mathbf{H}^1_f \big( 
G_{Z^o}, \pi(C,x)^\et
\big)
\ar[r]
&
\mathbf{H}^1_f \big( 
G_{K_\pP}, \pi(C,x)^\et
\big).
}
\]

\begin{ssProposition}
\label{DX66}
We put ourselves in the situation of paragraph \ref{DA41}, given by an open integer scheme $Z$, the category of Ind-lisse motives $\Dd_M(Z)$ over $Z$, $\Cc_M(Z) = \CAlg \Dd_M(Z)$, $x: C \to \one$ a morphism in $\Cc_M(Z)$, $\pP$ a closed point of $Z$ lying over the prime $p \in \ZZ$, and $Z^o \subset Z$ the complement of the set of all primes above $p$. We denote the descending central series of a group scheme by adding the decoration `$[n]$', numbered so that `$[1]$' corresponds to the abelianization. For each $n \ge 1$, the square \ref{DA41}(*) descends to a commuting square of presheaves of pointed sets
\[
\xymatrix{
\Aug^n(C)_\Qp \ar[r] \ar[d] & \Aug^n(C_{Z_\pP})_\Qp \ar[d]
\\
\mathbf{H}^1_f \big( 
G_{Z^o}, \pi(C,x)^{\et, [n]}
\big)
\ar[r]
&
\mathbf{H}^1_f \big( 
G_{K_\pP}, \pi(C,x)^{\et, [n]}
\big).
}
\]
\end{ssProposition}

\subsection{Comparison of filtrations}
\label{Sec:CompFil}
\label{XX66}
Recall that, by definition 
\[
\Aug^n(C) = \Aug(\fkS_nC) = \pi_0 \Aaug(\fkS_n C)
\]
is the presheaf of connected components of the space of augmentations of the $n$th step in the Iwanari tower of $C$. In proving Proposition \ref{DX66}, we take as evident, a commuting square
\[
\tag{*}
\xymatrix{
\Aug(\fkS_nC)_\Qp \ar[r] \ar[d] & \Aug(\fkS_nC_{Z_\pP})_\Qp \ar[d]
\\
\mathbf{H}^1_f \big( 
G_{Z^o}, \pi(\fkS_nC,x)^{\et}
\big)
\ar[r]
&
\mathbf{H}^1_f \big( 
G_{K_\pP}, \pi(\fkS_nC,x)^{\et}
\big).
}
\] 
What remains is to compare the sequence of maps
\[
\pi(\fkS_1C,x)^{\et}
\from
\pi(\fkS_2C,x)^{\et} 
\from 
\cdots
\from
\pi(C,x)^{\et}
\]
with the descending central series. To this end, we discuss a second tower, the \emph{1-minimal} tower, which serves as a bridge between the Sullivan tower on the one hand, and the descending central series of the unipotent fundamental group on the other hand. The proof of Proposition \ref{DX66} follows this discussion in paragraph \ref{DX66proof}.

\sspar{Definition of coherent 1-minimal tower}
\label{DB34}
As in paragraph \ref{BB44}, let $\Dd \in \CAlg (\PrLst)_{\Dd(F)/}$ be a presentably symmetric monoidal stable $F$-linear \oo-category over a field $F$ of characteristic 0 and let $\Cc = \CAlg(\Dd)$. Suppose given a t-structure on $\Dd$ with associated cohomological truncation functors $\tau^{\ge i}$, $\tau^{\le i}$. Fix a commutative algebra $C \in \Cc$ and assume $C$ is cohomologically connected with respect to the given t-structure ($H^i(C) \simeq 0$ for $i <0$, $H^0(C) \simeq \one$). We define the \emph{coherent 1-minimal tower of $C$} below left 
\[
\xymatrix
@ R = 4ex
{
C 
& V^\fkM_n \ar[r] & \fkM_n C \ar[r] & C \ar[r] & V^\fkM_n[1]
\\
\vdots \ar[u] && \fkM_{n+1}C  & \one \ar[l] 
\\ 
\fkM_2 C \ar[u] & \Sym ( \tau^{\le 2} V^\fkM_2) 
\ar[l] & 
\fkM_n C \ar[u] & 
\Sym (\tau^{\le 2}V^\fkM_n) 
\ar[l] \ar[u].
\\
\fkM_1 C \ar[u] & 
\Sym (\tau^{\le 2}V^\fkM_1) 
\ar[l] 
\\
\one \ar[u] & 
\Sym (\tau^{\le 2}V^\fkM_0) 
\ar[l]
}
\]
via the exact triangles in $\Dd$ and the coherent Hirsch extensions in $\Cc$ to the right. By construction, the coherent 1-minimal tower comes equipped with morphisms
\[
\tau^{\le 2} V_i^\fkM
\to
V_i
\quad \text{in } \Dd,
\qandq
\fkM_i C \to \fkS_i C
\quad \text{in } \Cc,  
\]
as well as a morphism of diagrams
\[
\tag{*}
\big\{\fkM_i C, \Sym (\tau^{\le 2} V_i^\fkM) \big\}_{i \ge 0}
\to
\big\{\fkS_i C, \Sym ( V_i) \big\}_{i \ge 0}
\]
in $\Cc$. 

\begin{ssExample}
\label{DB35}
Continuing with the situation and the notation of paragraph \ref{DB34}, suppose $\Dd = \Dd(F)$ is the derived \oo-category of $F$-\vectorspaces. Then $C$ may be modeled by a minimal cdga $M$, and the coherent 1-minimal tower may be modeled by a series of Hirsch extensions
\[
F \subset M_1 \subset M_2 \subset \cdots
\]
(which converges to $M$ only if $M$ is generated by $M^1$) known as the \emph{1-minimal filtration},
given as follows (see the discussion following Proposition 5.9 of Morgan \cite{Morgan}): $M_1$ is the cdg subalgebra generated by $Z^1M = H^1M$, and in terms of $M_i$, $M_{i+1}$ is the cdg subalgebra
\[
M_{i+1}=M_i \langle d_M\inv (M_i)^2 \rangle 
\subset M
\]
generated by $M_i + \set{x \in M^1}{dx \in (M_i)^2}$.

The first step $F \subset M_1$ is a Hirsch extension $F \subset F \langle \int W_0 \rangle$ in an evident way: we take $W_0$ to be $Z^1M$ placed in cohomological degree $2$ and we consider the zero map $W_0 \to F = M_0$. This means that $M_1 = \bigwedge W_0$ with $W_0$ placed in cohomological degree $1$,  and differential $d = 0$. The inclusion $Z^1M \subset M_1$ gives rise to an isomorphism of cdga's $F \langle W_0 \rangle \xto{\simeq} M_1$. 

Exhibiting the subsequent steps as Hirsch extensions requires arbitrary choices. For this, let $\Oo \subset M$ be any cdg subalgebra containing $Z^1M$ (our \textit{old} cdga), and let $\Nn = \Oo \langle d_M\inv \Oo^2 \rangle$ be the cdg subalgebra of $M$ generated by $\Oo + d_M\inv \Oo^2$ (our \textit{new} cdg subalgebra). We will exhibit $\Oo \to \Nn$ as a Hirsch extension. 

Let $W \subset Z^2\Oo$ be a vector subspace such that the projection $Z^2\Oo \surj H^2\Oo$ induces an isomorphism
\[
W \xto{\simeq} \ker ( H^2\Oo \to H^2 M).
\]
Then $W$ is contained in the image of the differential $M^1 \xto{d_M} M^2$, so we may choose a subspace $\int W \subset d_M\inv \Oo^2 \subset  M^1$ mapping isomorphically to $W$. 

Placing $W$ in cohomological degree $2$, we use the inclusion $W \subset \Oo$ to define the Hirsch extension
\[
\Oo \subset \Oo \langle \int W \rangle,
\]
and we use the inverse of $d_M: \int W \xto{\simeq} W$ to define a map 
\[
\Oo \langle \int W \rangle \to \Oo \langle d_M\inv \Oo^2 \rangle.
\]
Its injectivity follows from our assumption that the underlying graded-commutative algebra of $M$ is free, and the image evidently equals the cdg subalgebra of $M$ generated by $\Oo + \int W$.

For the surjectivity, consider $x \in M^1$ such that $dx \in \Oo^2$. Then there's an element $w \in W$ representing the class of $dx$ in $H^2\Oo$ so that
\[
dx = w + dy 
\]
for some $y \in \Oo^1$. Thus $d(x - y - \int w) = 0$ so 
\[
x =  y +z +\int w 
\]
for some $z \in Z^1M \subset \Oo$. We've succeeded in writing $x$ as a sum of an element of $\Oo$ and an element of $\int W$. 

It is now straightforward to check the equivalence of the 1-minimal filtration (*) and the coherent 1-minimal tower of Definition \ref{DB34}.
\end{ssExample}

\subsubsection{Review of Koszul duality of cdga's}
\label{PP33}
We summarize a few features of Koszul duality for cdga's over a field $F$ of characteristic 0 in the sense of Loday-Vallette \cite{LodayVallette}. Let $x:A \to F$ be an augmented cdga. Then there's an associated Koszul dual differential graded commutative Hopf algebra $B_x$, as well as an associated Koszul dual dg Lie coalgebra $L_x$. The dg Hopf algebra $B_x$ is again connected, and $H^0(B_x)$ retains the structure of a commutative Hopf algebra. We define \emph{the unipotent fundamental group of $A$} by 
\[
\pi(A,x) = \Spec H^0(B_x),
\]
a prounipotent $F$-group. There's a canonical isomorphism
\[
\Lie \pi(A,x) \simeq H^0(L_x)^\lor. 
\]
The constructions $B_x$, $L_x$ are functorial in $x:A \to F$.

If $M \to A$ is a 1-minimal model (denote the composition $M \to A \to F$ again by $x$), then the induced map of unipotent fundamental groups
\[
\pi(M,x) \from \pi(A,x)
\]
is an isomorphism.

If $A$ is 1-minimal, then the Koszul dual dg Lie coalgebra $L_x$ may be constructed in an elementary way. By definition, the underlying algebra $A \simeq \Sym(E)$ is an exterior algebra $A \simeq \bigwedge E$, and the differential defines a map
\[
d: E \to \bigwedge^2 E
\]
which makes $E$ into a Lie coalgebra; we define $L_x$ to be the Lie coalgebra $(E,d)$ placed in degree $0$. 

Returning to an augmented cdga $x:A \to F$ with 1-minimal model $M \to A$, let 
\[
\xymatrix{
F \ar[r] & \fkM_1 M \ar[drrr]_{x_1} \ar[r]  & \fkM_2 M  \ar[drr]^-{x_2} \ar[r]  & \cdots \ar[r] & M \ar[d]^{x}
\\
&&&& F
}
\]
be the 1-minimal filtration of $M$ with induced augmentations $x_i$ as shown. Then the induced sequence of Lie algebras 
\[
L_{x_1}^\lor \from L_{x_2}^\lor \from \cdots \from L_x^\lor
\]
is the sequence of quotients along the descending central series of $L_x^\lor$. This is inherent in the work of Sullivan \cite{Sullivan} and stated explicitly by Morgan \cite[paragraph 5.11]{Morgan}.

\subsubsection{Proof of Proposition \ref{DX66}}
\label{DX66proof}
After applying \'etale realization 
\[
\Re_\et: \Dd_M(Z) \to \Dd_\et(Z)
\]
(\ref{DA41}), we may use the standard t-structure on $\Dd_\et(Z)$ to form the 1-minimal tower. We may thus consider the map of towers \ref{DB34}(*) in $\Cc_\et(Z) := \CAlg \Dd_\et(Z)$: 
\[
\xymatrix{
\Qp  \ar[r] & \fkS_1 C_\et  \ar[r] & \fkS_2 C_\et  \ar[r] & \cdots  \ar[r] & C_\et 
\\
\Qp \ar@{=}[u] \ar[r] & 
\fkM_1 C_\et  \ar[r] \ar[u] & \fkM_2 C_\et  \ar[r] \ar[u] & \cdots  \ar[r] & C_\et \ar@{=}[u]
}
\]
Applying $\Spec H^0B_x(-)$ (where in each case, $x$ denotes the base-point induced by $x$), we obtain a diagram
\[
\xymatrix{
\pi(\fkS_1 C, x)^\et \ar[d] &
\pi(\fkS_2 C, x)^\et \ar[l] \ar[d] &
\cdots \ar[l] &
\pi(C, x)^\et \ar[l] \ar[d]
\\
\pi(\fkM_1 C, x)^\et &
\pi(\fkM_2 C, x)^\et \ar[l] &
\cdots \ar[l] &
\pi(C, x)^\et. \ar[l]
}
\]
of prounipotent $\Qp$-groups with $G_{Z^o}$-action and $G_{Z^o}$-equivariant maps. By paragraph \ref{PP33}, the bottom row agrees with the sequence
\[
\pi(C,x)^{\et, [1]} \from
\pi(C,x)^{\et, [2]} \from
\cdots \from
\pi(C,x)^\et
\]
of quotients along the descending central series of $\pi(C,x)^\et$. Restricting to $G_{K_\pP}$-representations, we get a similar diagram of prounipotent groups in crystalline $G_{K_\pP}$-representations. The global and local versions together give rise to commuting squares of presheaves of pointed sets (compatibly for varying $n$) 
\[
\xymatrix{
\mathbf{H}^1_f \big( 
G_{Z^o}, \pi(\fkS_nC,x)^{\et}
\big)
\ar[r] \ar[d]
&
\mathbf{H}^1_f \big( 
G_{K_\pP}, \pi(\fkS_nC,x)^{\et}
\big)
\ar[d]
\\
\mathbf{H}^1_f \big( 
G_{Z^o}, \pi(C,x)^{\et, [n]}
\big)
\ar[r]
&
\mathbf{H}^1_f \big( 
G_{K_\pP}, \pi(C,x)^{\et, [n]}
\big).
}
\]
We conclude by pasting the above to diagram \ref{XX66}(*). \qed 

\subsection{Example: a K-theoretic criterion for finiteness}
\label{Sec:Sample}

Let $Z$ be the complement of $s$ primes in $\Spec \ZZ$, let $X \to Z$ be the complement of the zero-section in an elliptic scheme of rank $r$, and let 
\[
k^{(2)}_2(X) = \dim_\QQ K^{(2)}_2(X).
\]
By paragraph \ref{EB17}, $\Aug^3(X)$ is representable by a finite type affine $\QQ$-variety of dimension
\[
\dim \Aug^3(X) = r + s + k^{(2)}_2(X). 
\]
By Kim \cite{kimii}, 
\[
\mathbf{H}^1\big(
\MFphi, \pi_1^{[n]}(X,x)^\dR_\Qp
\big)
\simeq F^0 \backslash \pi_1^{[n]}(X,x)^\dR_\Qp
\]
at any $Z$-integral base point $x$. By Beacom \cite{beacom2020computation}, 
\[
\dim F^0 \backslash \pi_1^{[3]}(X,x)^\dR_\Qp
=4.
\]
See Corwin \cite{CorwinEll} for a more detailed summary of this computation. Hence, whenever
\[
r + s + k^{(2)}_2(X) < 4,
\]
Theorem \ref{IN13} implies that there exists a nonzero function 
\[
f: \mathbf{H}^1 \big( \MFphi, \pi_1^{[3]}(X,x)^\dR_\Qp \big)
\to 
\AA^1_\Qp
\]
such that
\[
f \circ \ka_\Fphi \circ \la^a = 0.
\]
Again by Kim \cite{kimii}, the locally analytic map $ \ka_\Fphi \circ \de_\Fphi$ has dense image on every residue disk, so the locally analytic function $f \circ \ka_\Fphi \circ \de_\Fphi$ is nonzero on every residue disk, hence vanishes at only finitely many points. This shows that $X(Z)$ is finite. 

\printbibliography

\end{document}